\documentclass[oneside, 12pt]{amsart}

\usepackage{amsmath, amssymb, amsthm, amsfonts}

\ifdefined\directlua

\fi

\usepackage{mathabx,bbm}
\usepackage{microtype}
\usepackage{lmodern}
\usepackage[cal=euler,scr=rsfs]{mathalfa}
\usepackage[all]{xy}
\usepackage[a4paper, margin=1.1in]{geometry}
\usepackage{tikz}
\usepackage[backref=page]{hyperref}
\hypersetup{pdfborder={0 0 0}, colorlinks=true, urlcolor=blue, linkcolor=blue, citecolor=red}

\numberwithin{equation}{section}

\SelectTips{eu}{12}

\theoremstyle{theorem}
\newtheorem{theorem}{Theorem}[section]
\newtheorem{proposition}[theorem]{Proposition}
\newtheorem{lemma}[theorem]{Lemma}
\newtheorem{corollary}[theorem]{Corollary}

\theoremstyle{definition}
\newtheorem{definition}[theorem]{Definition}
\newtheorem{example}[theorem]{Example}

\theoremstyle{remark}
\newtheorem{remark}[theorem]{Remark}

\newcommand{\N}{\mathbb{N}}
\newcommand{\Z}{\mathbb{Z}}

\newcommand{\R}{\mathbb{R}}

\newcommand{\diam}{\operatorname{diam}}
\newcommand{\ind}{\operatorname{ind}}
\newcommand{\Fix}{\operatorname{Fix}}

\title{Uniform Lefschetz fixed-point theory}

\author{Tsuyoshi Kato}
\address{Department of Mathematics, Kyoto University, Kyoto 606-8502, Japan}
\email{tkato@math.kyoto-u.ac.jp}

\author{Daisuke Kishimoto}
\address{Faculty of Mathematics, Kyushu University, Fukuoka 819-0395, Japan}
\email{kishimoto@math.kyushu-u.ac.jp}

\author{Mitsunobu Tsutaya}
\address{Faculty of Mathematics, Kyushu University, Fukuoka 819-0395, Japan}
\email{tsutaya@math.kyushu-u.ac.jp}

\date{\today}



\begin{document}

\begin{abstract}
  We develop a Lefschetz fixed-point theory for uniformly continuous self-maps of noncompact
  manifolds of bounded geometry.
  For a uniform, simply-connected, complete Riemannian manifold $M$ of bounded geometry and a
  uniformly continuous map $f\colon M\to M$ satisfying $d(f,1)<\infty$, we define a uniform
  Lefschetz class $\mathscr{L}(f)$ and show that $\mathscr{L}(f)=0$ if and only if $f$ is
  uniformly homotopic to a map having neither fixed points nor fixed points at infinity; an
  analogous obstruction class characterizes the existence of a nonvanishing bounded uniformly
  continuous vector field on $M$.
  Both classes take values in the $0$-th uniformly finite homology of Block and Weinberger,
  and hence vanish automatically when $M$ is nonamenable.
  When $M$ is a Galois covering of a closed manifold, we further establish a
  Lefschetz-Hopf-type formula for $\mathscr{L}(f)$ and an extension of the Poincar\'e-Hopf
  theorem for vector fields.
  The key new tool is \emph{uniform bounded cohomology}, a variant of Gromov's bounded
  cohomology for metric spaces satisfying uniform homotopy invariance, excision and Poincar\'e
  duality.
  As a further application, we show that uniform and Lipschitz homotopy theory coincide for
  uniform simplicial complexes; combined with our obstruction theory, this allows us to obtain
  Lipschitz analogues of the above results.
\end{abstract}

\maketitle


\section{Introduction}

\subsection{Classical Lefschetz fixed-point theorem}

The Lefschetz fixed-point theorem states that for a self-map $f\colon X\to X$ of a finite complex $X$, if the Lefschetz number $L(f)\ne 0$, then $f$ has a fixed-point. The converse was established by Fadell \cite{F} for simply-connected closed manifolds: if $X$ is a simply-connected closed manifold, then $L(f)=0$ if and only if $f$ is homotopic to a fixed-point-free map. Note that the compactness of $X$ is essential in both defining the Lefschetz number and proving the theorem because the Lefschetz number counts the number of fixed-points with multiplicities. Hence generalizing the Lefschetz fixed point theorem to noncompact manifolds is a challenging problem.

\subsection{Setting for noncompact theory}

Let $M$ be a connected noncompact manifold. If a self-map of $M$ has isolated fixed-points, then these fixed-points can be swept out to infinity by a homotopy, yielding a new self-map of $M$ that is fixed-point-free (cf. \cite{KKT1,W}). To avoid this triviality, we require homotopies not to send points to infinity. Uniformly continuous homotopies - briefly, \emph{uniform homotopies} - satisfy this requirement. Hence we work with uniform homotopies and, consequently, with uniformly continuous self-maps of $M$. Suppose that a map $f\colon M\to M$ satisfies $d(x,f(x))\to\infty$ as $x$ goes to infinity, where $d$ is a metric on $M$. Then fixed-points of $f$ are contained in a compact subset $C\subset M$, and the fixed-point problem reduces to the classical coincidence problem between the restriction $f\vert_C$ and the inclusion of $C\to M$. A more precise formulation of this setting is given in \cite{N,S}. For a self-map $g\colon X\to X$ of a metric space $(X,d)$, set
\[
  d(g,1)=\sup_{x\in X}d(g(x),x),
\]
where $1$ denotes the identity map.
Hence we work with uniformly continuous self-maps $f\colon M\to M$ satisfying $d(f,1)<\infty$. Such boundedness conditions on maps and homotopies are standard in the study of noncompact manifolds. Following \cite{R1,R2,W}, we also assume that $M$ has bounded geometry.

We now describe the connections between our setting and previous work. Roughly speaking, there are three cases in which the Lefschetz fixed-point theorem has been generalized to noncompact manifolds:

\begin{itemize}
  \item The fixed-point set is contained in a compact subset \cite{N,S}.

  \item Manifolds are equipped with a cocompact proper action of a discrete group and self-maps are equivariant \cite{LR}.

  \item Self-maps are $C^1$-close to the identity map \cite{W}.
\end{itemize}

\noindent As mentioned above, the first case reduces to the classical compact case, and hence our theory does not need to cover it. We consider maps that are not necessarily equivariant, even when manifolds are equipped with group action; therefore our setting differs from the second case. However, assuming a nice group action on a manifold, we will obtain the Lefschetz-Hopf theorem relating the Lefschetz class to the indices of fixed-points. Weinberger \cite{W} established a criterion for the existence of a unit-length vector field with bounded derivative in terms of the Euler class in bounded cohomology. If such a vector field exists, then perturbing the identity map by this vector field yields a fixed-point-free self-map that can be made arbitrarily $C^1$-close to the identity. This gives the precise statement of the third case. Since this perturbation is achieved by a uniform homotopy, and self-maps $f$ that are $C^1$-close to the identity are uniformly continuous and satisfy $d(f,1)<\infty$, our results encompass the third case.

\subsection{Fixed-point at infinity}

Let $f\colon M\to M$ be a self-map of a noncompact manifold $M$. Suppose there is a sequence $x_1,x_2,\ldots\in M$ such that $d(x_n,f(x_n))<1/n$ for all $n\ge 1$. If this sequence has a limit point, then the limit point is a fixed-point of $f$. If, on the other hand, the sequence has no limit point, it diverges to infinity as $n\to\infty$. In this sense, one may regard the sequence as indicating the presence of a \emph{fixed-point at infinity}. Motivated by this observation, we consider the (non)existence of fixed-points at infinity in addition to the usual fixed-points. Hence we introduce the notion of strong fixed-point-freeness as the property of a self-map having neither usual fixed-points nor fixed-points at infinity.

\begin{definition}
  A self-map $f\colon X\to X$ of a metric space $(X,d)$ is called \emph{strongly fixed-point-free} if there is $\epsilon>0$ satisfying $d(f(x),x)\ge\epsilon$ for all $x\in M$.
\end{definition}

\subsection{Uniform Lefschetz fixed-point theorem}

Let $M$ be a simply-connected complete Riemannian manifold of bounded geometry satisfying a mild condition called \emph{uniformness} (Definition \ref{uniform manifold}). For example, any Riemannian manifold admitting an isometric cocompact action of a discrete group is uniform. Let $f\colon M\to M$ be a uniformly continuous self-map satisfying $d(f,1)<\infty$.
We define the \emph{uniform Lefschetz class} $\mathscr{L}(f)$ by means of a certain obstruction class, in analogy with Fadell's construction in the compact case \cite{F}.
It will turn out that $\mathscr{L}(f)$ actually belongs to the $0$-th uniformly finite homology $H_0^\mathrm{uf}(M)$ introduced by Block and Weinberger \cite{BW}.
Uniformly finite homology is a quasi-isometry invariant of metric spaces, and can be used to characterize amenability in metric spaces.
Although uniformly finite homology is generally difficult to compute in higher dimensions, the $0$-th uniformly finite homology can be computed when a metric space admits a quasi-lattice, e.g. a properly discontinuous and cocompact action of a group, in the sense of \cite{BW}.

Now we state the main theorem.

\begin{theorem}
  [Theorem \ref{main late}]
  \label{main}
  Let $M$ be a uniform simply-connected complete Riemannian manifold of bounded geometry with $\dim M\ge 2$. Let $f\colon M\to M$ be a uniformly continuous map satisfying $d(f,1)<\infty$. Then $\mathscr{L}(f)=0$ if and only if $f$ is uniformly homotopic to a strongly fixed-point-free uniformly continuous map.
\end{theorem}

\begin{remark}
  If a manifold has dimension $1$, we cannot define the uniform Lefschetz class for technical reasons. However, any simply-connected $1$-dimensional manifold of bounded geometry is biLipschitz homeomorphic to $\R$ with the Euclidean metric (See Example \ref{triangulation}). It is straightforward to see that, by applying a shift, every uniformly continuous map $f\colon\R\to\R$ with $d(f,1)<\infty$ can be uniformly homotoped to a strongly fixed-point-free map, where $\R$ is equipped with the Euclidean metric. Therefore, in the $1$-dimensional case, the uniform Lefschetz class need not be considered.
\end{remark}

As noted above, the uniform Lefschetz class can be viewed as an element of the $0$-th uniformly finite homology.
As mentioned above, this homology, in turn, is closely related to the amenability of a metric space in the sense of Block and Weinberger \cite{BW}.
As a consequence, we obtain a result for nonamenable manifolds (Corollary \ref{amenable Lefschetz}). In particular, we have the following.

\begin{corollary}
  [Corollary \ref{corollary}]
  Let $M$ be the universal cover of a closed connected manifold with nonamenable $\pi_1$. Then any uniformly continuous map $f\colon M\to M$ satisfying $d(f,1)<\infty$ is uniformly homotopic to a strongly fixed-point-free uniformly continuous map.
\end{corollary}

\begin{remark}
  As in Example \ref{connected sum}, one can construct a self-map of the universal cover of a closed connected manifold with amenable $\pi_1$ that is not uniformly homotopic to any strongly fixed-point-free uniformly continuous map.
\end{remark}

\subsection{Lefschetz-Hopf theorem}

The classical Lefschetz-Hopf theorem states that if a map $f\colon X\to X$ of a closed manifold $X$ has isolated fixed-points, then
\[
  L(f)=\sum_{x\in\Fix(f)}\ind_x(f)
\]
where $\ind_x(f)$ denotes the local index of $f$ at a fixed-point $x\in\Fix(f)$. We extend the Lefschetz-Hopf theorem to noncompact manifolds. However, for a self-map of a noncompact manifold, the set of isolated fixed-points may be infinite, making the above summation over local indices ill-defined. In \cite{KKT1,KKT2}, the authors introduced a new method for counting infinitely many points on a noncompact manifold that arises as a Galois covering of a closed manifold. We adopt their approach to establish the following version of the Lefschetz-Hopf theorem.

If a manifold $M$ admits a properly discontinuous and cocompact action of a group $G$, then $H_0^\mathrm{uf}(M)$ can be identified with the coinvariants of the module of bounded functions $G\to\Z$. In particular, the Lefschetz class is represented by a bounded map $G\to\Z$.

\begin{theorem}
  [Theorem \ref{Lefschetz-Hopf late}]
  \label{Lefschetz-Hopf}
  Let $M$ be the universal cover of a closed connected manifold $N$, and let $\pi=\pi_1(N)$. Suppose $\pi$ is amenable. Then for a strongly tame (see Definition \ref{tame}) uniformly continuous map $f\colon M\to M$, the uniform Lefschetz class $\mathscr{L}(f)$ is represented by the map
  \[
    \pi\to\Z,\quad g\mapsto\sum_{x\in gK\cap\Fix(f)}\ind_x(f)
  \]
  where $K$ is a fundamental domain.
\end{theorem}

Using Theorem \ref{Lefschetz-Hopf}, we will, in Example \ref{connected sum} below, easily construct a self-map of the infinite connected sum of $S^2\times S^2$ that is not uniformly homotopic to a strongly fixed-point-free map. We also obtain the following corollary (cf. \cite{N,S}).

\begin{corollary}
  [Corollary \ref{finitely many fixed-points}]
  \label{finite fixed-points}
  Let $M$ be as in Theorem \ref{Lefschetz-Hopf}, and let $f\colon M\to M$ be a uniformly continuous map satisfying $d(f,1)<\infty$. If $M$ is noncompact and $f$ has only finitely many fixed-points, then $f$ is uniformly homotopic to a strongly fixed-point-free map.
\end{corollary}

Consider the map
\[
  f\colon\R^n\to\R^n,\quad x\mapsto 2x,
\]
where the Euclidean space $\R^n$ is the universal cover of the $n$-torus, whose fundamental group is $\Z^n$, an amenable group. Clearly, $f$ is uniformly continuous and has a single fixed-point, namely $0\in\R^n$. In contrast to Corollary \ref{finite fixed-points}, $f$ is never uniformly homotopic to a strongly fixed-point-free map. This phenomenon is caused by the fact that $d(f,1)=\infty$. Nevertheless, this situation is covered by the results of \cite{N,S}. This example suggests that there should exist a fundamentally different Lefschetz fixed-point theory for self-maps that satisfy $d(f,1)=\infty$, possibly allowing infinitely many fixed-points, and encompassing the theories developed in \cite{N,S}.

\subsection{Uniform bounded (co)homology}

We now explain how the Lefschetz fixed-point theory is developed in the uniformly continuous setting. As mentioned above, the uniform Lefschetz class $\mathscr{L}(f)$ is defined in terms of a certain obstruction class. However, this obstruction class does not, in general, belong to ordinary cohomology, since ordinary cohomology behaves poorly for noncompact manifolds and fails to distinguish uniform homotopy equivalences from general homotopy equivalences. To overcome these difficulties, in Section \ref{Uniform bounded cohomology}, we introduce a new cohomology for metric spaces, called \emph{uniform bounded cohomology}. It is defined using singular cochains that are bounded on families of singular simplices uniformly distributed in a metric space. This cohomology is well-suited for noncompact manifolds of bounded geometry and is invariant under uniform homotopy.

In Section \ref{Uniform bounded homology}, we also introduce \emph{uniform bounded homology} as the natural dual of uniform bounded cohomology; this coincides with Engel's uniformly locally finite homology \cite{E}. The terminology change is deliberate, as it emphasizes the duality with uniform bounded cohomology. There is a natural map from uniform bounded homology into Block and Weinberger's uniformly finite homology \cite{BW}, which is an isomorphism in dimension $0$ when a metric space is geodesic. As in Section \ref{Poincare duality section}, within uniform bounded homology, one can define the fundamental class of an oriented connected manifold of bounded geometry. Remarkably, Poincar\'e duality holds in this setting; for an oriented connected manifold of bounded geometry, taking the cap product with the fundamental class gives an isomorphism between uniform bounded homology and uniform bounded cohomology.

Uniform bounded cohomology may be viewed as a variant of Gromov's bounded cohomology.
Unlike the latter, uniform bounded cohomology satisfies excision, allowing us to construct a cellular uniform bounded cohomology for a \emph{uniform CW complex} as in Section \ref{Uniform bounded cellular (co)homology}.
Using this framework, in Section \ref{Obstruction theory}, we develop an obstruction theory based on uniform bounded cellular cohomology and define the uniform Lefschetz class via an obstruction class in uniform bounded cohomology associated with the Fadell-Neuwirth fibration together with the Poincar\'e duality.
From this construction, as in Section \ref{Uniform Lefschetz fixed-point theorem}, Theorem \ref{main} follows directly from the obstruction theory.
We then generalize Theorem \ref{main} to the coincidence theorem.

In \cite{KKT1}, the authors developed a method for counting infinitely many points on a Galois covering of a closed manifold in order to establish the Poincar\'e-Hopf theorem for such coverings.
In Section \ref{Localization}, we adopt this approach and prove Theorem \ref{Lefschetz-Hopf} by localizing the obstruction class at the fixed-points.
We then return to the Poincar\'e-Hopf theorem for Galois coverings of closed manifolds and improve it in the uniformly continuous setting.

\subsection{Lipschitz homotopy theory}

Here we remark on the difference between uniform homotopy theory and Lipschitz homotopy theory.
Gromov \cite{G2} studied estimates of minimal Lipschitz constants within each homotopy classes.
In particular, he proved and conjectured several estimates for the growth of minimal Lipschitz
constants of maps from spheres as a function of the degree $d$.
For example, the minimal Lipschitz constant of a degree $d$ map $f\colon S^2\to S^2$ can be
bounded above by $(\mathrm{const})\cdot\sqrt{d}$.
Gromov also asked how the Lipschitz constant of a null-homotopy of a Lipschitz null-homotopic
map depends on the Lipschitz constant of the map itself; this question has since developed
into an active area of quantitative homotopy theory, with polynomial bounds established by
Chambers, Manin, and Weinberger \cite{CMW}, a general ``shadowing principle'' relating
Sullivan models to genuine Lipschitz maps due to Manin \cite{Manin}, and further refinements
and extensions in \cite{CDMW,Berd,BM,Hansen} (see \cite{Guth} for a survey).

Ferry and Weinberger \cite{FW} studied a more qualitative aspect of the same circle of ideas:
the mere \textit{existence}, rather than quantitative control, of Lipschitz homotopies.
They formulated the obstruction to the existence of a Lipschitz null-homotopy of a certain map
$\widetilde M\to S^n$ associated to a closed manifold $M$ as a class in the uniformly finite
(bounded) cohomology of the universal cover $\widetilde M$, and showed that this obstruction
vanishes if and only if $\pi_1(M)$ is nonamenable; this is exactly the same type of phenomenon
--- amenability controlling the vanishing of a bounded cohomology class --- that appears in our
obstruction theory below. The proof given there is only an outline of a possible strategy,
without full details.
A related, purely qualitative approach to passing between uniform and Lipschitz homotopies was
given by Liu, Yu, and Liu \cite{LYL}, who used the modulus of continuity of a given uniform
homotopy to bound the Lipschitz constant of an associated Lipschitz homotopy, for a large class
of target spaces including compact Riemannian manifolds.

In this paper, we prove a simplicial approximation theorem for uniformly continuous maps
between uniform simplicial complexes, without any compactness assumption on the target
(Theorem \ref{theorem_uniformsimplicialapproximation} in the appendix).
As a byproduct, we obtain a general comparison between uniform and Lipschitz homotopy theory:
for uniform simplicial complexes $K,K'$, every uniformly continuous map $K\to K'$ is uniformly
homotopic to a Lipschitz map, and moreover, two Lipschitz maps $K\to K'$ that are uniformly
homotopic are already homotopic through a Lipschitz homotopy (Corollary
\ref{corollary equivalence uniform Lipschitz}).
In other words, uniform homotopy theory and Lipschitz homotopy theory of uniform simplicial
complexes coincide; this extends the result of Liu, Yu, and Liu \cite{LYL}, whose argument
requires the target to be compact.
Using this comparison theorem, we give, as an application of our obstruction theory, a proof in
full detail of the theorem of Ferry and Weinberger.

\subsection*{Acknowledgement}

The authors were partially supported by JSPS KAKENHI Grant numbers JP23K22394 (Kato), JP22K03284 (Kishimoto), and JP22K03317 and JP26K06810 (Tsutaya).


\section{Uniform bounded cohomology}\label{Uniform bounded cohomology}

In this section, we introduce a new cohomology theory for metric spaces, called uniform bounded cohomology, which can be considered as a variant of Gromov's bounded cohomology. We show that uniform bounded cohomology is a uniform homotopy invariant and, unlike bounded cohomology, it satisfies the Mayer-Vietoris sequence and excision.

\subsection{Uniform bounded cochain}

Recall from \cite{G} that a bounded $n$-cochain on a space $X$ is an $n$-cochain $u$ satisfying
\[
  \sup_\sigma|u(\sigma)|<\infty
\]
where $\sigma$ ranges over all singular $n$-simplices in $X$. The collection of bounded cochains forms a cochain complex, and its cohomology is Gromov's bounded cohomology. By definition, bounded cohomology is functorial with respect to all continuous maps and is invariant under ordinary homotopy. However, it does not satisfy excision, so cellular cohomology cannot be formulated in terms of bounded cohomology. In particular, we cannot develop obstruction theory based on the skeletal filtration using bounded cohomology. To address this, we define a new variant of bounded cohomology that is invariant under uniform homotopy and satisfies excision.

Let $\Delta^n$ denote the standard $n$-simplex with its usual metric, and let $X$ be a metric space. We say that a family of singular $n$-simplices $\{\sigma_i\colon\Delta^n\to X\}_{i\in I}$ is said to be \emph{uniform} if the following conditions hold:

\begin{enumerate}
  \item For any $r>0$, there exists $K_r\ge 0$ such that the ball $B_r(x)$ of radius $r$ around any point $x\in X$ intersects at most $K_r$ singular simplices $\sigma_i$.

  \item The family $\{\sigma_i\}_{i\in I}$ is equicontinuous.
\end{enumerate}

Note that since $\Delta^n$ is compact, a family of singular simplices $\{\sigma_i\colon\Delta^n\to X\}_{i\in I}$ is equicontinuous if and only if it is uniformly equicontinuous; that is, for every $\epsilon>0$, there exists $\delta>0$ such that for all $x,y\in\Delta^n$ with $d(x,y)<\delta$, one has $d(\sigma_i(x),\sigma_i(y))<\epsilon$ for all $i\in I$.

\begin{example}
  \label{bounded space 1}
  Let $X$ be a bounded metric space, and let $\{\sigma_i\colon\Delta^n\to X\}_{i\in I}$ be a uniform family of singular $n$-simplices in $X$. If $r$ is larger than the diameter of $M$, then by local finiteness, the family $\{\sigma_i\}_{i\in I}$ consists of at most $K_r$ simplices. Therefore, $\{\sigma_i\}_{i\in I}$ is finite.
\end{example}

\begin{example}
  \label{uniform simplicial complex}
  A \emph{uniform simplicial complex} is a (geometric) simplicial complex satisfying the following conditions:

  \begin{enumerate}
    \item The piecewise Euclidean metric is assigned so that every edge has length $1$.

    \item There exists $K\ge 0$ such that each vertex belongs to at most $K$ simplices.
  \end{enumerate}

  \noindent Let $S$ denote the set of inclusions of $n$-simplices of a uniform simplicial complex. The first condition ensures that $S$ is equicontinuous. For any $r>0$, these conditions guarantee the existence of $L_r\ge 0$ such that the ball $B_r(x)$ of radius $r$ around any point $x\in X$ contains at most $L_r$ vertices. By the second condition, $B_r(x)$ therefore meets at most $L_r\cdot K$ simplices. Hence $S$ is uniform.
\end{example}

A \emph{pseudo-normed abelian group} is a pair $(A,|\cdot|)$ consisting of an abelian group $A$ and a function $|\cdot|\colon A\to\R_{\ge 0}$, called a pseudo-norm, satisfying
\[
  |0|=0\quad\text{and}\quad|a+b|\le|a|+|b|
\]
for $a,b\in A$. We denote by $\underline{A}$ a pseudo-normed abelian group having the underlying group $A$. 

\begin{example}
  The group of integers $\Z$ equipped with the standard absolute value is a pseudo-normed abelian group. In what follows, we only consider this standard pseudo-norm on $\Z$.
\end{example}

\begin{example}
  \label{fin gen}
  The product $\Z^n$ carries a pseudo-norm
  \[
    \Z^n\to\R_{\ge 0},\quad(x_1,\ldots,x_n)\mapsto|x_1|+\cdots|x_n|.
  \]
  Let $A$ be a finitely generated abelian group. Fix an isomorphism $A/\mathrm{Tor}\cong\Z^n$. Then $A$ inherits a pseudo-norm via the composite
  \[
    A\to A/\mathrm{Tor}\cong\Z^n\to\R_{\ge 0},
  \]
  where the last map is the pseudo-norm above. It is straightforward to verify for any subset $S\subset A$, the boundedness of $S$ with respect to this pseudo-metric does not depend on the choice of the isomorphism $A/\mathrm{Tor}\cong\Z^n$. Hence, any finitely generated abelian group $A$ will be equipped with the pseudo-norm induced in this way. This convention is justified since we are concerned only with the boundedness of subsets of $A$.
\end{example}

Let $X$ be a metric space, and let $\underline{A}$ be a pseudo-normed abelian group. A \emph{uniform bounded $n$-cochain} on $X$ with coefficients in $\underline{A}$ is an $n$-cochain $u$ on $X$ with values in the underlying abelian group $A$ of $\underline{A}$ such that
\[
  \sup_{i\in I}|u(\sigma_i)|<\infty
\]
for any uniform family of singular $n$-simplices $\{\sigma_i\colon\Delta^n\to X\}_{i\in I}$. Let $C^n_\mathrm{ub}(X;\underline{A})$ denote the abelian group of all uniform bounded $n$-cochains on $X$ with coefficients in $\underline{A}$. Observe that if $\{\sigma_i\colon\Delta^n\to X\}_{i\in I}$ is a uniform family of singular $n$-simplices, then $\{\partial_0\sigma_i,\ldots,\partial_n\sigma_i\colon\Delta^{n-1}\to X\}_{i\in I}$ is also a uniform family of singular $(n-1)$-simplices, where $\partial_j$ denote the $j$-th face operator. Hence, for any uniform family $\{\sigma_i\colon\Delta^n\to X\}_{i\in I}$ and $u\in C^{n-1}_\mathrm{ub}(X;\underline{A})$, we have
\[
  \sup_{i\in I}|(\delta u)(\sigma_i)|=\sup_{i\in I}|u(\partial\sigma_i)|\le\sup_{i\in I}\sum_{j=0}^n|u(\partial_j\sigma_i)|<\infty,
\]
where $\delta$ and $\partial$ denote the coboundary and boundary operators, respectively. Therefore, $C^*_\mathrm{ub}(X;\underline{A})$ forms a cochain complex.

\begin{definition}
  The \emph{uniform bounded cohomology} of a metric space $X$ with coefficients in a pseudo-normed abelian group $\underline{A}$ is defined by
  \[
    H^*_\mathrm{ub}(X;\underline{A})=H^*(C^*_\mathrm{ub}(X;\underline{A})).
  \]
\end{definition}

We will abbreviate $H^*_\mathrm{ub}(X;\underline{\Z})$ simply as $H^*_\mathrm{ub}(X)$.


\subsection{Uniform homotopy invariance}

 The uniformness of families of singular simplices in a metric space is not preserved under arbitrary continuous maps. Consequently, unlike bounded cohomology, not every continuous map induces homomorphism in uniform bounded cohomology. To address this, we restrict our attention to a special class of maps that do induce well-defined maps in uniform bounded cohomology.

 \begin{definition}
   [Block and Weinberger \cite{BW}]
   A map $f\colon X\to Y$ between metric spaces is called \emph{effectively proper} if for every $r>0$, there exists $s>0$ such that if $d(f(x),f(y))<r$ for $x,y\in X$, then $d(x,y)<s$.
 \end{definition}

Block and Weinberger \cite{BW} introduced uniformly finite homology for metric spaces and showed that effectively proper bornologous maps (in their terminology, effectively proper Lipschitz maps) induce well-defined maps in uniformly finite homology. In the present setting, we will consider effectively proper uniformly continuous maps between metric spaces. 

\begin{example}
  \label{EPUC example}
  \begin{enumerate}
    \item The inclusion of a subspace of a metric space is an effectively proper and uniformly continuous.

    \item Every biLipschitz homeomorphism is an effectively proper uniformly continuous maps.

    \item If $f\colon X\to Y$ is a uniformly continuous map between metric spaces, then the map
    \[
      X\to X\times Y,\quad x\mapsto(x,f(x))
    \]
    is effectively proper and uniformly continuous.
  \end{enumerate}
\end{example}

\begin{lemma}
  \label{EPUC uniform}
  Let $f\colon X\to Y$ be an effectively proper uniformly continuous map between metric spaces. If $\{\sigma_i\colon\Delta^n\to X\}_{i\in I}$ is a uniform family of singular $n$-simplices, then $\{f\circ\sigma_i\colon\Delta^n\to Y\}_{i\in I}$ is also a uniform family of singular $n$-simplices.
\end{lemma}

\begin{proof}
  Since $f$ is uniformly continuous, the family $\{f\circ\sigma_i\}_{i\in I}$ is equicontinuous. Take any $r>0$ and $y\in Y$. Let
  \[
    I(y,r)=\{i\in I\mid B_r(y)\text{ intersects }f\circ\sigma_i\}.
  \]
  If $I(y,r)\ne\emptyset$, then there exists $x\in X$ such that $d(f(x),y)<r$, hence
  \[
    I(y,r)\subset I(f(x),2r).
  \]
  For every $i\in I(y,r)$, choose $x_i\in\sigma_i(\Delta^n)$ such that $d(y,f(x_i))<r$, so that $d(f(x),f(x_i))<2r$. Since $f$ is effectively proper, there exists $s>0$ such that $d(x,x_i)<s$ for all $i\in I(y,r)$. Note that $s$ only depends on $r$, which is independent of the choice of $y$ and $x$. Since the family $\{\sigma_i\}_{i\in I}$ is uniform, there exists $K_s>0$ such that $B_s(x)$ meets at most $K_s$ simplices $\sigma_i$. Consequently, $|I(y,r)|\le K_s$. If $I(y,r)=\emptyset$, then trivially $|I(y,r)|<K_s$ as well. Therefore, $\{f\circ\sigma_i\}_{i\in I}$ is uniform.
\end{proof}

\begin{proposition}
  \label{induced map}
  Let $f\colon X\to Y$ be an effectively proper uniformly continuous map between metric spaces. Then the induced map
  \[
    f^*\colon H^*_\mathrm{ub}(Y;\underline{A})\to H^*_\mathrm{ub}(X;\underline{A}),\quad[u]\mapsto[u\circ f]
  \]
  is well defined.
\end{proposition}

\begin{proof}
  By Lemma \ref{EPUC uniform}, if $u$ is a uniform bounded cochain on $Y$, then $u\circ f$ is a uniform bounded cochain on $X$. Thus, $f$ induces a cochain map \[
    f^\sharp\colon C^*_\mathrm{ub}(Y;\underline{A})\to C^*_\mathrm{ub}(X;\underline{A}),\quad u\mapsto u\circ f
  \]
  and consequently a homomorphism in cohomology, as claimed.
\end{proof}

Clearly, the identity map of a metric space is effectively proper and uniformly continuous, and the composition of effectively proper uniformly continuous maps is again effectively proper and uniformly continuous. Hence, metric spaces together with effectively proper uniformly continuous maps form a category, denoted by $\mathbf{Met}_\mathrm{EPUC}$. By Proposition \ref{induced map}, for any pseudo-normed abelian group $\underline{A}$, we obtain a functor
\[
  \mathbf{Met}_\mathrm{EPUC}^\mathrm{op}\to\mathbf{GrAb},\quad X\mapsto H_\mathrm{ub}^*(X;\underline{A}),
\]
where $\mathbf{GrAb}$ denotes the category of graded abelian groups.

Let $(X,Y)$ be a pair of metric spaces, that is, a metric space $X$ together with a subspace $Y\subset X$. Let $\underline{A}$ be a pseudo-normed abelian group. As noted in Example \ref{EPUC example}, the inclusion $Y\to X$ is effectively proper and uniformly continuous. Thus as in the proof of Proposition \ref{induced map}, it induces a cochain map $C_\mathrm{ub}^*(X;\underline{A})\to C_\mathrm{ub}^*(Y;\underline{A})$. We define the cochain complex $C_\mathrm{ub}^*(X,Y;\underline{A})$ by its kernel. Then we define the uniform bounded cohomology of the pair $(X,Y)$ by
\[
  H^*_\mathrm{ub}(X,Y;\underline{A})=H^*(C_\mathrm{ub}^*(X,Y;\underline{A})).
\]
By this construction, we immediately obtain the following.

\begin{theorem}
  \label{cohomology exact sequence}
  Let $(X,Y)$ be a pair of metric spaces. Then there is a long exact sequence
  \[
    \cdots\to H^n_\mathrm{ub}(X,Y;\underline{A})\to H^n_\mathrm{ub}(X;\underline{A})\to H^n_\mathrm{ub}(Y;\underline{A})\to H^{n+1}_\mathrm{ub}(X,Y;\underline{A})\to\cdots.
  \]
\end{theorem}

We next establish the uniform homotopy invariance of uniform bounded cohomology.

\begin{definition}
  Let $f,g\colon X\to Y$ be maps between metric spaces. We say that $f$ is \emph{uniformly homotopic} to $g$ if there exists a uniformly continuous map $H\colon X\times[0,1]\to Y$, called a \emph{uniform homotopy}, such that $H(-,0)=f$ and $H(-,1)=g$.
\end{definition}

Note that if $f,g\colon X\to Y$ are uniformly homotopic, then both $f$ and $g$ are necessarily uniformly continuous.

\begin{remark}
  The above definition of a uniform homotopy differs from that of Calder and Siegel \cite{CS} who impose a weaker requirement: for every $\epsilon>0$, there exists $\delta>0$ such that for all $x\in X$, $d(H(x,s),H(x,t))<\epsilon$ whenever $|s-t|<\delta$.
\end{remark}

\begin{lemma}
  \label{EPUC homotopy}
  Let $f,g\colon X\to Y$ be uniformly homotopic maps between metric spaces. If $f$ is effectively proper, then any uniform homotopy from $f$ to $g$ is also effectively proper. In particular, $g$ is effectively proper whenever $f$ is.
\end{lemma}

\begin{proof}
  Let $H\colon X\times[0,1]\to Y$ be a uniform homotopy from $f$ to $g$. Then there exists $\delta>0$ such that $d(H(x,s),H(y,t))<1$ whenever $\sqrt{d(x,y)^2+(s-t)^2}<\delta$. In particular, for any $x\in X$, we have $d(H(x,s),H(x,t))<1$ whenever $|s-t|<\delta$. Therefore,
  \begin{align*}
    d(H(x,0),H(x,s))&\le\sum_{i=1}^nd(H(x,\tfrac{(i-1)\delta}{2}),H(x,\tfrac{i\delta}{2}))+d(H(x,\tfrac{n\delta}{2}),H(x,s))\\
    &\le n+1\\
    &\le\frac{2}{\delta}+1
  \end{align*}
  where $n$ is the largest integer satisfying $0\le\tfrac{n\delta}{2}\le s$. Similarly, $d(H(x,t),H(x,1))\le\frac{2}{\delta}+1$. Hence, for any $x,y\in X$ and $s,t\in[0,1]$,
  \begin{align*}
    d(f(x),f(y))&\le d(H(x,0),H(x,s))+d(H(x,s),H(y,t))+d(H(y,t),H(y,0))\\
    &<d(H(x,s),H(y,t))+\frac{4}{\delta}+2.
  \end{align*}
  Suppose $f$ is effectively proper. Then there exists $s_0>0$ such that $d(H(x,s),H(y,t))<r$ implies $d(x,y)<s_0$. Since $\delta$ is fixed, we have $d((x,s),(y,t))<\sqrt{s_0^2+1}$, showing that $H$ is effectively proper. The assertion follows immediately.
\end{proof}

\begin{theorem}
  \label{homotopy invariance cohomology}
  Let $f,g\colon X\to Y$ be effectively proper uniformly continuous maps between metric spaces. If $f$ and $g$ are uniformly homotopic, then
  \[
    f^*=g^*\colon H^*_\mathrm{ub}(Y;\underline{A})\to H^*_\mathrm{ub}(X;\underline{A}).
  \]
\end{theorem}

\begin{proof}
  Let $H\colon X\times[0,1]\to Y$ be a uniform homotopy between $f$ and $g$. It suffices to show that the induced cochain maps $f^\sharp,g^\sharp\colon C^*_\mathrm{ub}(Y;\underline{A})\to C^*_\mathrm{ub}(X;\underline{A})$ are chain homotopic. By Lemma \ref{EPUC homotopy}, the uniform homotopy $H$ is effectively proper. Arguing as in the proof of Lemma \ref{EPUC uniform}, one verifies that the prism operator associated with $H$ on the singular chain complex preserves uniform families of singular simplices. Consequently, the standard prism construction defines a chain homotopy $P\colon C^*_\mathrm{ub}(Y;\underline{A})\to C^{*-1}_\mathrm{ub}(X;\underline{A})$ satisfying
  \[
    g^\sharp-f^\sharp=\delta P+P\delta.
  \]
  Hence, $f^\sharp$ and $g^\sharp$ are chain homotopic.
\end{proof}

\subsection{Mayer-Vietoris sequence and excision}

Let $X$ be a metric space, and let $\underline{A}$ be a pseudo-normed abelian group. For $\epsilon>0$, denote by $C_{\mathrm{ub},\epsilon}^n(X;\underline{A})$ the abelian group of singular $n$-cochains that are bounded on all uniform families of singular $n$-simplices whose diameters are less than $\epsilon$. As in the definition of $C^*_\mathrm{ub}(X;\underline{A})$, one easily verifies that $C_{\mathrm{ub},\epsilon}^*(X;\underline{A})$ forms a cochain complex under the usual coboundary operator.

We now establish a technical but fundamental lemma showing that restricting attention to simplices of sufficiently small diameter does not alter the cohomology.

\begin{lemma}
  \label{C_epsilon}
  The restriction
  \[
    \iota\colon C^*_\mathrm{ub}(X;\underline{A})\to C_{\mathrm{ub},\epsilon}^*(X;\underline{A})
  \]
  is a chain homotopy equivalence.
\end{lemma}

\begin{proof}
  Let $\{\sigma_i\colon\Delta^n\to X\}_{i\in I}$ be a uniform family of singular $n$-simplices. Since the family is equicontinuous, there exists $\delta>0$ such that for any subset $A\subset\Delta^n$ with $\diam(A)<\delta$, the image $\sigma_i(A)$ has diameter $<\epsilon$ for all $i\in I$. Since $\Delta^n$ is compact, for every $i\in I$ there exists an integer $m_i\ge 0$ such that every $n$-simplex appearing in the $m$-th barycentric subdivision of $\sigma$ has diameter $<\delta$. Denote by $m_i$ the smallest such integer. By equicontinuity, the set $\max\{m_i\mid i\in I\}$ is bounded. Hence, the family consisting of all $n$-simplices arising from the $m_i$-th barycentric subdivision of $\sigma_i$ for all $i\in I$ forms a uniform family. Let $S^{m_i}(\sigma_i)$ denote the sum of all $n$-simplices in the $m_i$-th barycentric subdivision of the $n$-simplex $\sigma_i$. It satisfies
  \[
    \partial S^{m_i}(\sigma_i)=S^{m_i}(\partial\sigma_i),
  \]
  where $S^{m_i}$ on the right denotes the linear extension of $S^{m_i}$ to chains. This yields a cochain map
  \[
    \rho\colon C_{\mathrm{ub},\epsilon}^*(X;\underline{A})\to C_\mathrm{ub}^*(X;\underline{A}),\quad(\rho(u))(\sigma_i)=u(S^{m_i}(\sigma_i))
  \]
  for $u\in C_{\mathrm{ub},\epsilon}^*(X;\underline{A})$. By construction, $\iota\circ\rho=1$. Moreover, by an argument analogous to that of \cite[Proposition 2.21]{H}, one verifies that $\rho\circ\iota$ is chain homotopic to the identity map of $C^*_\mathrm{ub}(X;\underline{A})$. Hence $\iota$ is a chain homotopy equivalence.
\end{proof}

Let $\mathcal{U}=\{U_i\}_{i\in I}$ be a cover of a metric space $X$. We say that $\mathcal{U}$ has a \emph{Lebesgue number} $\epsilon>0$ if any subset of $X$ with diameter less than $\epsilon$ is contained in some $U_i$. In the proof of the Mayer-Vietoris sequence and excision for singular cohomology, the barycentric subdivision plays the central role. There, the number of subdivisions applied to a singular simplex depends on the simplex itself. In the present uniform setting, this procedure may fail to preserve uniform families of simplices. Therefore, to establish the Mayer-Vietoris sequence and excision results for uniform bounded cohomology, we will impose an additional condition on the Lebesgue number of the covering $\mathcal{U}=\{U_i\}_{i\in I}$.

\begin{theorem}
  \label{Mayer-Vietoris cohomology}
  Let $X=U\cup V$ be a cover of a metric space $X$ with Lebesgue number $\epsilon>0$. Then there is a long exact sequence
  \begin{multline*}
    \cdots\to H_\mathrm{ub}^n(X;\underline{A})\xrightarrow{(i_U^*,i_V^*)}H_\mathrm{ub}^n(U;\underline{A})\oplus H_\mathrm{ub}^n(V;\underline{A})\\
    \xrightarrow{j_U^*-j_V^*}H_\mathrm{ub}^n(U\cap V;\underline{A})\to H_\mathrm{ub}^{n+1}(X;\underline{A})\to\cdots
  \end{multline*}
  where $i_U\colon U\to X$, $i_V\colon V\to X$, $j_U\colon U\cap V\to U$ and $j_U\colon U\cap V\to V$ are inclusions.
\end{theorem}

\begin{proof}
  Since the cover $X=U\cup V$ admits a Lebesgue number $\epsilon>0$, every singular simplex of diameter less than $\epsilon$ is contained entirely in either $U$ or $V$. Hence we obtain a short exact sequence of cochain complexes;
  \[
    0\to C_{\mathrm{ub},\epsilon}^*(X;\underline{A})\xrightarrow{(i_U^\sharp,i_V^\sharp)}C_{\mathrm{ub},\epsilon}^*(U;\underline{A})\oplus C_{\mathrm{ub},\epsilon}^*(V;\underline{A})\xrightarrow{j_U^\sharp-j_V^\sharp}C_{\mathrm{ub},\epsilon}^*(U\cap V;\underline{A})\to 0.
  \]
  Note that the inclusion $C_{\mathrm{ub},\epsilon}^*(X;\underline{A})\to C_\mathrm{ub}^*(X;\underline{A})$ is natural with respect to the metric space $X$. Therefore, by Lemma \ref{C_epsilon}, we obtain an analogous short exact sequence for $C_\mathrm{ub}^*$, and the desired long exact sequence in cohomology follows.
\end{proof}

\begin{theorem}
  \label{excision cohomology}
  Let $(X,Y)$ be a pair of metric spaces. Suppose that $Z\subset Y$ satisfies $B_\epsilon(Z)\subset Y$ for some $\epsilon>0$. Then the natural map
  \[
    H^*_\mathrm{ub}(X,Y;\underline{A})\to H^*_\mathrm{ub}(X-Z,Y-Z;\underline{A})
  \]
  is an isomorphism.
\end{theorem}

\begin{proof}
  Set $U=X-Z$ and $V=Y$. Then $X=U\cup V$ is a cover with Lebesgue number $\epsilon>0$. Consequently,  $C^*_{\mathrm{ub},\epsilon}(X;\underline{A})=C^*_{\mathrm{ub},\epsilon}(U;\underline{A})+C^*_{\mathrm{ub},\epsilon}(V;\underline{A})$. Hence the natural map
  \[
    \mathrm{Ker}\{C_{\mathrm{ub},\epsilon}^*(X;\underline{A})\to C_{\mathrm{ub},\epsilon}^*(Y;\underline{A})\}\to\mathrm{Ker}\{C_{\mathrm{ub},\epsilon}^*(U;\underline{A})\to C_{\mathrm{ub},\epsilon}^*(U\cap V;\underline{A})\}
  \]
  is an isomorphism. Moreover, we have a commutative diagram
  \[
    \xymatrix{
      C_\mathrm{ub}^*(X,Y;\underline{A})\ar[r]\ar[d]&C_\mathrm{ub}^*(U,U\cap V;\underline{A})\ar[d]\\
      \mathrm{Ker}\{C_{\mathrm{ub},\epsilon}^*(X;\underline{A})\to C_{\mathrm{ub},\epsilon}^*(Y;\underline{A})\}\ar[r]&\mathrm{Ker}\{C_{\mathrm{ub},\epsilon}^*(U;\underline{A})\to C_{\mathrm{ub},\epsilon}^*(U\cap V;\underline{A})\}.
    }
  \]
  By Lemma \ref{C_epsilon}, both vertical maps are chain homotopy equivalences. Therefore, the top horizontal map induces an isomorphism in cohomology, proving the statement.
\end{proof}


\section{Uniform bounded homology}\label{Uniform bounded homology}

In this section, we introduce uniform bounded homology for metric spaces and observe that it satisfies properties analogous to those of uniform bounded cohomology. We also examine the $0$-th uniform bounded homology in detail.

\subsection{Uniform bounded homology}

Let $X$ be a metric space, and let $\underline{A}$ be a pseudo-normed abelian group. A (possibly infinite) formal linear combination
\[
  \sum_{i\in I}a_i\sigma_i
\]
with coefficients $a_i\in A$ and singular $n$-simplices $\sigma_i$ in $X$, is called a \emph{uniform bounded $n$-chain} if the family $\{\sigma_i\colon\Delta^n\to X\}_{i\in I}$ is uniform and
\[
  \sup_{i\in I}|a_i|<\infty.
\]
Let $C_n^\mathrm{ub}(X;\underline{A})$ be the abelian group of all uniform bounded $n$-chains in $X$ with coefficients in $\underline{A}$. Quite similarly to $C^*_\mathrm{ub}(X;\underline{A})$, one verifies that $C_*^\mathrm{ub}(X;\underline{A})$ is a chain complex with boundary operator
\[
  \partial\left(\sum_{i\in I}a_i\sigma_i\right)=\sum_{i\in I}a_i\partial\sigma_i.
\]

\begin{definition}
  The \emph{uniform bounded homology} of a metric space $X$ with coefficients in a pseudo-normed abelian group $\underline{A}$ is defined by
  \[
    H_*^\mathrm{ub}(X;\underline{A})=H_*(C_*^\mathrm{ub}(X;\underline{A})).
  \]
\end{definition}

We abbreviate $H_*^\mathrm{ub}(X;\underline{\Z})$ by $H_*^\mathrm{ub}(X)$. Here, $H_*^\mathrm{ub}(X)$ coincides with the uniformly locally finite homology of $X$ introduced by Engel \cite{E}. We intentionally adopt a different name in order to emphasize that this theory is the natural dual of uniform bounded cohomology (see Section \ref{Poincare duality section} for details).

By Lemma \ref{EPUC uniform}, we obtain the following.

\begin{proposition}
  If $f\colon X\to Y$ is an effectively proper uniformly continuous map between metric spaces, then the induced map
  \[
    f_*\colon H_*^\mathrm{ub}(X;\underline{A})\to H_*^\mathrm{ub}(Y;\underline{A}),\quad\left[\sum_{i\in I}a_i\sigma_i\right]\mapsto\left[\sum_{i\in I}a_i(f\circ\sigma_i)\right]
  \]
  is well defined.
\end{proposition}

Let $(X,Y)$ be a pair of metric spaces. We define
\[
  C_*^\mathrm{ub}(X,Y;\underline{A})=C_*^\mathrm{ub}(X;\underline{A})/C_*^\mathrm{ub}(Y;\underline{A}).
\]
Then $C_*^\mathrm{ub}(X,Y;\underline{A})$ is a chain complex, and we define the uniform bounded homology of $(X,Y)$ by
\[
  H_*^\mathrm{ub}(X,Y;\underline{A})=H_*(C_*^\mathrm{ub}(X,Y;\underline{A}))
\]

\begin{theorem}
  \label{homology exact sequence}
  Let $(X,Y)$ be a pair of metric spaces, and let $\underline{A}$ be a pseudo-normed abelian group. Then there is a long exact sequence
  \[
    \cdots\to H_n^\mathrm{ub}(Y;\underline{A})\to H_n^\mathrm{ub}(X;\underline{A})\to H_n^\mathrm{ub}(X,Y;\underline{A})\to H_{n-1}^\mathrm{ub}(Y;\underline{A})\to\cdots.
  \]
\end{theorem}

\begin{proof}
  By definition, there is a short exact sequence of chain complexes
  \[
    0\to C_*^\mathrm{ub}(Y;\underline{A}) \to C_*^\mathrm{ub}(X;\underline{A})\to C_*^\mathrm{ub}(X,Y;\underline{A})\to 0.
  \]
  The desired long exact sequence in homology follows immediately.
\end{proof}

The proof of Theorem \ref{homotopy invariance cohomology} applies verbatim to uniform bounded homology, yielding the following.

\begin{theorem}
  \label{homotopy invariance homology}
  Let $f,g\colon X\to Y$ be effectively proper uniformly continuous maps between metric spaces. If $f$ and $g$ are uniformly homotopic, then
  \[
    f_*=g_*\colon H^*_\mathrm{ub}(X;\underline{A})\to H^*_\mathrm{ub}(Y;\underline{A}).
  \]
\end{theorem}

Let $X$ be a metric space, and let $\underline{A}$ be a pseudo-normed abelian group. For $\epsilon>0$, we define $C^{\mathrm{ub},\epsilon}_n(X;\underline{A})$ denote the subgroup of $C^{\mathrm{ub}}_n(X;\underline{A})$ consisting of chains supported on singular $n$-simplices of diameter $<\epsilon$. Then $C^{\mathrm{ub},\epsilon}_n(X;\underline{A})$ is a subcomplex of $C^\mathrm{ub}_n(X;\underline{A})$. Analogously to Lemma \ref{C_epsilon}, we obtain the following.

\begin{lemma}
  \label{C^epsilon}
  The inclusion
  \[
    C_*^{\mathrm{ub},\epsilon}(X;\underline{A})\to C_*^\mathrm{ub}(X;\underline{A})
  \]
  is a chain homotopy equivalence.
\end{lemma}

The proof of Theorem \ref{Mayer-Vietoris cohomology} applies verbatim to uniform bounded homology, using Lemma \ref{C^epsilon} in place of Lemma \ref{C_epsilon}. Hence we obtain the following Mayer-Vietoris sequence.

\begin{theorem}
  \label{Mayer-Vietoris homology}
  Let $X=U\cup V$ be a cover of a metric space $X$ with Lebesgue number $\epsilon>0$. Then there is a long exact sequence
  \begin{multline*}
    \cdots\to H^\mathrm{ub}_n(U\cap V;\underline{A})\xrightarrow{((j_U)_*,(j_V)_*)}H^\mathrm{ub}_n(U;\underline{A})\oplus H^\mathrm{ub}_n(V;\underline{A})\\
    \xrightarrow{i_U^*-i_V^*}H^\mathrm{ub}_n(X;\underline{A})\to H_\mathrm{ub}^{n-1}(U\cap V;\underline{A})\to\cdots
  \end{multline*}
  where $i_U\colon U\to X$, $i_V\colon V\to X$, $j_U\colon U\cap V\to U$ and $j_U\colon U\cap V\to V$ are inclusions.
\end{theorem}

Bounded uniform homology also satisfies excision.

\begin{theorem}
  \label{excision homology}
  Let $(X,Y)$ be a pair of metric spaces. If $Z\subset Y$ satisfies $B_\epsilon(Z)\subset Y$ for some $\epsilon>0$, then the natural map
  \[
    H_*^\mathrm{ub}(X-Z,Y-Z;\underline{A})\to H_*^\mathrm{ub}(X,Y;\underline{A})
  \]
  is an isomorphism.
\end{theorem}

\begin{proof}
  Let $U=X-Z$ and $V=Y$. Then the cover $X=U\cup V$ has Lebesgue number $\epsilon>0$, implying that $C^{\mathrm{ub},\epsilon}_*(X;\underline{A})=C_*^{\mathrm{ub},\epsilon}(U;\underline{A})+C_*^{\mathrm{ub},\epsilon}(V;\underline{A})$. In particular,
  \[
    C^{\mathrm{ub},\epsilon}_*(X,Y;\underline{A})=C_*^{\mathrm{ub},\epsilon}(U;\underline{A})/C_*^{\mathrm{ub},\epsilon}(U\cap V;\underline{A}).
  \]
  The conclusion follows from Lemma \ref{C^epsilon}.
\end{proof}

\subsection{$0$-th uniform bounded homology}

We recall the definition of Block and Weinberger's uniformly finite homology \cite{BW}. Let $X$ be a metric space. Define $C_n^\mathrm{uf}(X)$ to be the abelian group of (possibly infinite) formal linear combinations
\[
  \sum_{i\in I}a_i\bar{x}_i,
\]
where $a_i\in\Z$ and $\bar{x}_i\in X^{n+1}=\overbrace{X\times\cdots\times X}^{n+1}$, satisfying the following conditions:

\begin{enumerate}
  \item For every $r>0$, there exists $K_r\ge 0$ such that for any $\bar{y}\in X^{n+1}$, the number of $\bar{x}_i$ intersecting $B_r(\bar{y})$ is at most $K_r$.

  \item There exists $R>0$ such that $d(\bar{x}_i,\Delta)<R$ for all $i\in I$, where $\Delta\subset X^{n+1}$ denotes the diagonal subset.

  \item $\displaystyle\sup_{i\in I}|a_i|<\infty$.
\end{enumerate}

\noindent Then $C_*^\mathrm{uf}(X;\underline{A})$ is a chain complex under the boundary operator
\[
  \partial\colon C_n^\mathrm{uf}(X)\to C_{n-1}^\mathrm{uf}(X),\quad\sum_{i\in I}a_i\bar{x}_i\mapsto\sum_{i\in I}a_i(\partial\bar{x}_i)
\]
where for $(x_0,\ldots,x_n)\in X^{n+1}$, we set
\[
  \partial(x_0,\ldots,x_n)=\sum_{i=0}^n(-1)(x_0,\ldots,x_{i-1},x_{i+1},\ldots,x_n).
\]

\begin{definition}
  [Block and Weinberger \cite{BW}]
  The \emph{uniformly finite homology} of a metric space $X$ is defined by
  \[
    H^\mathrm{uf}_*(X)=H_*(C_*^\mathrm{uf}(X)).
  \]
\end{definition}

As mentioned above, uniformly finite homology is functorial with respect to effectively proper bornologous maps and is invariant under a quasi-isometry. Thus, it detects large scale (coarse) properties of metric spaces.

Let $e_0,\ldots,e_n$ denote the standard basis of $\R^{n+1}$, and regard $\Delta^n$ as the convex hull of $e_0,\ldots,e_n$. Given a uniform family $\{\sigma_i\colon\Delta^n\to X\}_{i\in I}$ of singular $n$-simplices in $X$, the family of $(n+1)$-tuples $\{\sigma_i(e_0),\ldots,\sigma_i(e_n)\}_{i\in I}$ satisfies conditions (1) and (2) above, since equicontinuity of $\{\sigma_i\}_{i\in I}$ implies uniform control on the diameter of the image of the vertices. Hence, we may define a chain map
\[
  \alpha\colon C_n^\mathrm{ub}(X)\to C_n^\mathrm{uf}(X),\quad\sum_{i\in I}a_i\sigma_i\mapsto\sum_{i\in I}a_i(\sigma_i(e_0),\ldots,\sigma_i(e_n)).
\]
This induces a natural map on homology
\[
  \alpha\colon H_*^\mathrm{ub}(X)\to H_*^\mathrm{uf}(X).
\]

\begin{proposition}
  \label{0-dim}
  If $X$ is a geodesic space, i.e. any two points in $X$ can be joined by a geodesic, then the map
  \[
    \alpha\colon H_0^\mathrm{ub}(X)\to H_0^\mathrm{uf}(X)
  \]
  is an isomorphism.
\end{proposition}

\begin{proof}
  By definition, the map $\alpha\colon C_0^\mathrm{ub}(X)\to C_0^\mathrm{uf}(X)$ is an isomorphism. Let
  \[
    \sum_{i\in I}a_i(x_i,y_i)\in C_1^\mathrm{uf}(X)
  \]
  for $a_i\in\Z$ and $(x_i,y_i)\in X^2$. Since $X$ is a geodesic space, for each $i\in I$, there exists a geodesic $\gamma_i\colon[0,1]\to X$ from $x_i$ to $y_i$. By condition (2) of the definition of $C_*^\mathrm{uf}(X)$, there exists $R>0$ such that $d((x_i,y_i),\Delta)<R$ for all $i\in I$; in particular, $d(x_i,y_i)<2R$. Thus for each $i$, we have
  \[
    d(\gamma_i(s),\gamma_i(t))<2R|s-t|,
  \]
  showing that the family $\{\gamma_i\}_{i\in I}$ is equicontinuous. For any $r>0$, there exists $K_r\ge 0$ such that for every $z\in X^2$, the number of $(x_i,y_i)$ for $i\in I$ intersecting $B_r(z)$ is at most $K_r$. If $(x_i,y_i)$ does not intersect $B_s(y)$, then $\gamma_i$ does not intersect $B_{s-2R}(y)$. Hence, the number of $\gamma_i$ intersecting $B_{r+R}(y)$ is at most $K_r$. Therefore, the family $\{\gamma_i\}_{i\in I}$ is uniform, and the correspondence $(x_i,y_i)\mapsto\gamma_i$ induces a right inverse chain map, completing the proof.
\end{proof}

In \cite{BW}, Block and Weinberger introduced a notion of amenability for metric spaces that generalizes the classical concept of amenability for groups. We recall the definition of amenable metric spaces. Let $X$ be a metric space, and let $N_\epsilon(A)$ denotes the open $\epsilon$-neighborhood of $A\subset X$. A subset $\Gamma\subset X$ is called a \emph{quasi-lattice} if the following conditions hold:

\begin{enumerate}
  \item There exists $c>0$ such that $N_c(\Gamma)=X$.

  \item For every $r>0$, there exists $K_r>0$ such that for all $x\in X$, $|\Gamma\cap B_r(x)|\le K_r$.
\end{enumerate}

\noindent A metric space $X$ is said to be \emph{amenable} if it has a quasi-lattice $\Gamma$ such that for every $r,\delta>0$, there is a finite subset $F\subset\Gamma$ satisfying
\[
  \frac{|N_r(F)\cap N_r(X-F)\cap\Gamma|}{|F|}<\delta.
\]

\begin{example}
  \label{amenable group}
  A finitely generated group equipped with a word metric is amenable as a metric space if and only if it is amenable as a group.
\end{example}

\begin{example}
  \label{Galois cover space}
  Let $G\to X\to Y$ be a Galois covering, where $X$ is a proper length space on which $G$ acts isometrically and $Y$ is compact. Then $X$ is amenable if and only if $G$ is amenable.
\end{example}

Block and Weinberger \cite[Theorem 3.1]{BW} established a homological characterization of amenability in terms of the $0$-th uniformly finite homology.

\begin{proposition}
  \label{amenable H_0}
  A metric space $X$ is amenable if and only if $H_0^\mathrm{uf}(X)\ne 0$.
\end{proposition}

By Examples \ref{amenable group} and \ref{Galois cover space} and Proposition \ref{amenable H_0}, we get:

\begin{corollary}
  \label{H_0 amenable}
  Let $G\to X\to Y$ be as in Example \ref{Galois cover space}. Then $H_0^\mathrm{uf}(X)\ne 0$ if and only if $G$ is amenable.
\end{corollary}

We show a useful algebraic description of $H_0^\mathrm{uf}(X)$ for the metric space $X$ in Example \ref{Galois cover space}. For a set $S$, let $\ell^\infty(S)$ denote the abelian group of bounded functions $S\to\Z$. If $G$ is a group, then $G$ acts on $\ell^\infty(G)$ by right translation:
\[
  (g\cdot\phi)(x)=\phi(xg)
\]
for $x,g\in G$ and $\phi\in\ell^\infty(G)$. The module of coinvariants is defined by
\[
  \ell^\infty(G)_G=\ell^\infty(G)/\langle\phi-g\cdot\phi\mid\phi\in\ell^\infty(G),\,g\in G\rangle.
\]
We record a property of $\ell^\infty(G)_G$ that we will use later.

\begin{proposition}
  \label{coinvariants finite}
  Let $G$ be a finitely generated infinite group. If a function $f\colon G\to\Z$ satisfies $f(x)=0$ for all but finitely many $x\in G$, then $[f]=0$ in $\ell^\infty(G)_G$.
\end{proposition}

\begin{proof}
  Combine \cite[Proposition 2.3]{KKT1} and \cite[Lemma 7.7]{Wh}.
\end{proof}

Let $G\to X\to Y$ be as in Example \ref{Galois cover space}. Then there exist $R>0$ and a decomposition
\[
  X=\coprod_{g\in G}X_g,
\]
where $g\in X_g$ and $\mathrm{diam}(X_g)<R$ for all $g\in G$.

\begin{lemma}
  \label{H_0 covering}
  Let $G\to X\to Y$ be as in Example \ref{Galois cover space}. Then the map
  \[
    H_0^\mathrm{uf}(X)\to\ell^\infty(G)_G,\quad \left[\sum_{i\in I}a_ix_i\right]\mapsto\left[G\to\Z,\quad g\mapsto\sum_{i\in I_g}a_i\right]
  \]
  is a well-defined isomorphism, where $I_g=\{i\in I\mid x_i\in X_g\}$.
\end{lemma}

\begin{proof}
  In \cite{BW}, it is proved that uniformly finite homology is invariant under quasi-isometry. The inclusion $G\to X$ is a quasi-isometry, and the map
  \[
    \rho\colon X\to G,\quad\rho(X_g)=g
  \]
  is a quasi-isometry inverse. Therefore, the induced map
  \[
    H_0^\mathrm{uf}(X)\to H_0^\mathrm{uf}(G),\quad\left[\sum_{i\in I}a_ix_i\right]\mapsto\left[\left(\sum_{i\in I_g}a_i\right)g\right]
  \]
  is an isomorphism. Moreover, by \cite{BNW}, the map
  \[
    C_0^\mathrm{uf}(G)\to\ell^\infty(G),\quad\sum_{g\in G}b_gg\mapsto[G\to\Z,\quad g\mapsto b_g]
  \]
  induces an isomorphism $H_0^\mathrm{uf}(G)\xrightarrow{\cong}\ell^\infty(G)_G$. Composing these isomorphisms yields the desired statement.
\end{proof}


\section{Uniform bounded cellular (co)homology}\label{Uniform bounded cellular (co)homology}

In this section, we introduce the notion of a uniform CW complex, namely a metric space with a regular CW decomposition whose metric is uniformly approximated by the metrics on the closed discs corresponding to its cells. We then define the uniform bounded (co)homology of such complexes and show that it agrees naturally with uniform bounded (co)homology defined in Sections \ref{Uniform bounded cohomology} and \ref{Uniform bounded homology}.

\subsection{Uniform CW complex}

Let $X$ be a CW complex, and let $\Lambda$ denote the set of its cells. We regard $\Lambda$ as an index set of characteristic maps of $X$. Write $\Lambda_n$ for the subset of $n$-cells and $\Lambda_{\le n}=\bigcup_{i=0}^n\Lambda_i$. Let $D^n\subset\R^n$ and $S^{n-1}\subset\R^n$ denote the unit disc and unit sphere, each equipped with the standard Euclidean metric.

\begin{definition}
  \label{uniform CW complex}
  Let $X$ be a metric space. A regular CW decomposition of $X$ with characteristic maps $\{\phi_\lambda\colon D^{n_\lambda}\to X\}_{\lambda\in\Lambda}$ is called \emph{uniform} if the following conditions hold:

  \begin{enumerate}
    \item For every $r>0$, there exists $K_r\ge 0$ such that for any $x\in X$, at most $K_r$ cells meet $B_r(x)$.

    \item The families of the characteristic maps $\{\phi_\lambda\}_{\lambda\in\Lambda}$ and their inverse maps $\{\phi_\lambda^{-1}\colon\phi_\lambda(D^{n_\lambda})\to D^{n_\lambda}\}_{\lambda\in\Lambda}$ are equicontinuous.

    \item There exists an integer $L>0$ such that for every $\epsilon>0$, there exists $\delta>0$ with the following property: whenever $x,y\in X$ satisfy $d(x,y)<\delta$, there are indices $\lambda_i\in\Lambda$ and points $z_i,\bar{z}_i\in D^{n_{\lambda_i}}$ ($i=1,\ldots,L$) such that
    \[
      d_{D^{n_{\lambda_1}}}(z_1,\bar{z}_1)+\cdots+d_{D^{n_{\lambda_L}}}(z_L,\bar{z}_L)<\epsilon
    \]
    and
    \[
      x=\phi_{\lambda_1}(z_1),\quad\phi_{\lambda_i}(\bar{z}_i)=\phi_{\lambda_{i+1}}(z_{i+1})\quad(i=1,\ldots,L-1),\quad\phi_{\lambda_L}(\bar{z}_L)=y.
    \]
  \end{enumerate}
\end{definition}

A metric space equipped with a specific uniform regular CW decomposition is called a \emph{uniform CW complex}. Conditions (1) and (2) ensure that for each $n\ge 0$, the family of the characteristic maps $\{\phi_\lambda\colon D^n\to X\}_{\lambda\in\Lambda_n}$ is a uniform family of singular $n$-simplices in $X$ (after fixing a biLipschitz identification $\Delta^n\cong D^n$). Condition (3) guarantees that the ambient metric on $X$ is uniformly approximated by the intrinsic metrics on the discs corresponding to its cells.

\begin{example}
  Every uniform simplicial complex (as defined in Section \ref{Uniform bounded cohomology}) is a uniform CW complex.
\end{example}

\begin{example}
  \label{triangulation}
  By \cite{BDG,B}, a (possibly noncompact) Riemannian manifold of bounded geometry is biLipschitz homeomorphic to a uniform simplicial complex. Hence any such manifold admits a uniform regular CW decomposition.
\end{example}

\begin{example}
  The standard CW decomposition
  \[
    [0,1]=\{0\}\cup(0,1)\cup\{1\}
  \]
  is uniform, and we will always regard $[0,1]$ with this uniform structure.
\end{example}

It is well known that a map out of a CW complex is continuous if and only if its compositions with all characteristic maps are continuous. For uniform CW complexes, we can prove the analogous statement for uniform continuity.

\begin{lemma}
  \label{continuity}
  Let $X$ be a uniform CW complex with characteristic maps $\{\phi_\lambda\colon D^{n_\lambda}\to X\}_{\lambda\in\Lambda}$, and let $f\colon X\to Y$ be a map into a metric space $Y$. Then $f$ is uniformly continuous if and only if the family $\{f\circ\phi_\lambda\}_{\lambda\in\Lambda}$ is equicontinuous.
\end{lemma}

\begin{proof}
  The “only if” part is immediate. We prove the “if” part. Let $L>0$ be as in Definition \ref{uniform CW complex}, and assume that $\{f\circ\phi_\lambda\}_{\lambda\in\Lambda}$ is equicontinuous. Then the family is, in particular, uniformly equicontinuous: for every $\epsilon>0$, there exists $\rho>0$ such that for all $\lambda\in\Lambda$ and $a,b\in D^{n_\lambda}$, $d(f\circ\phi_\lambda(a),f\circ\phi_\lambda(b))<\epsilon$ if $d_{D^{n_\lambda}}(a,b)<\rho$. By Definition \ref{uniform CW complex} (3), there exists $\delta>0$ such that whenever $x,y\in X$ satisfy $d(x,y)<\delta$, we can find $\lambda_i\in\Lambda$ and $z_i,\bar{z}_i\in D^{n_{\lambda_i}}$ ($i=1,\ldots,L$) with
  \[
    d_{D^{n_{\lambda_1}}}(z_1,\bar{z}_1)+\cdots+d_{D^{n_{\lambda_L}}}(z_L,\bar{z}_L)<\rho
  \]
  and
  \[
    x=\phi_{\lambda_1}(z_1),\quad\phi_{\lambda_i}(\bar{z}_i)=\phi_{\lambda_{i+1}}(\bar{z}_{i+1})\quad(i=1,\ldots,L-1),\quad\phi_{\lambda_L}(\bar{z}_L)=y.
  \]
  Using the triangle inequality, we obtain
  \[
    d(f(x),f(y))\le\sum_{i=1}^Ld((f\circ\phi_{\lambda_i})(z_i),(f\circ\phi_{\lambda_i})(\bar{z}_i))<L\epsilon.
  \]
  Hence $f$ is uniformly continuous.
\end{proof}

Next, we introduce the notion of uniform discreteness, which plays a role analogous to local finiteness in the cellular setting. A metric space $X$ is said to be \emph{uniformly discrete} if the following conditions hold:
\begin{enumerate}
  \item for every $r>0$, there exists $K_r\ge 0$ such that for any $x\in X$, the ball $B_r(x)$ contains at most $K_r$ points of $X$.

  \item $\displaystyle\inf_{x\ne y\in X}d(x,y)>0$.
\end{enumerate}

\begin{example}
  Let $X$ be a uniform CW complex with index set $\Lambda$. If we identify $\Lambda$ with the subset of $X$ consisting of the centers of its cells, then for each $n\ge 0$, $\Lambda_n$ is uniformly discrete.
\end{example}

\subsection{Uniform bounded cellular (co)homology}

For a simplicial complex $X$, let $X_n$ denote the $n$-skeleton of $X$ for $n\ge 0$, and set $X_{-1}=\emptyset$. For a set $S$ and a pseudo-normed abelian group $\underline{A}$, we write $\ell^\infty(S;\underline{A})$ for the abelian group of bounded functions $S\to A$.

\begin{lemma}
  \label{H(V,W)}
  For $n\ge 0$, let $\Lambda$ be a uniformly discrete metric space, and define
  \[
    V_n=\Lambda\times\mathrm{Int}(D^n)\quad\text{and}\quad W_n=\Lambda\times\mathrm{Int}(D^n-0)
  \]
  where the metric of $V_n$ is given by
  \[
    d((\lambda,x),(\mu,y))=
    \begin{cases}
      d(x,y)&(\lambda=\mu)\\
      d(\lambda,\mu)&(\lambda\ne\mu)
    \end{cases}
  \]
  and $W_n$ is considered as a subspace of $V_n$. Then there are isomorphisms
  \[
    H_\mathrm{ub}^*(V_n,W_n;\underline{A})\cong H^\mathrm{ub}_*(V_n,W_n;\underline{A})\cong
    \begin{cases}
      \ell^\infty(\Lambda;\underline{A})&(*=n)\\
      0&(*\ne n).
    \end{cases}
  \]
\end{lemma}

\begin{proof}
  We prove the cohomological statement; the homological one follows verbatim. The projection $V_n\to\Lambda$ is an effectively proper uniform homotopy equivalence. Hence, by Theorem \ref{homotopy invariance cohomology}, $H^*_\mathrm{ub}(V_n;\underline{A})\cong H^*_\mathrm{ub}(\Lambda_n;\underline{A})$. By direct computation from the definition of uniform bounded cochains,
  \begin{equation}
    \label{V_n}
    H^*_\mathrm{ub}(V_n;\underline{A})\cong
    \begin{cases}
      \ell^\infty(\Lambda;\underline{A})&(*=0)\\
      0&(*\ne 0).
    \end{cases}
  \end{equation}
  The long exact sequence for the pair $(V_n,W_n)$ (Theorem \ref{cohomology exact sequence}) gives
  \begin{equation}
    \label{(V,W)}
    \cdots\to H^*_\mathrm{ub}(V_n,W_n;\underline{A})\to H^*_\mathrm{ub}(V_n;\underline{A})\to H^*_\mathrm{ub}(W_n;\underline{A})\to H^{*+1}_\mathrm{ub}(V_n,W_n;\underline{A})\to\cdots.
  \end{equation}
  It therefore suffices to determine $H^*_\mathrm{ub}(W_n;\underline{A})$. We claim that
  \[
    H^*_\mathrm{ub}(W_n;\underline{A})\cong
    \begin{cases}
      \ell^\infty(\Lambda;\underline{A})&(*=0)\\
      \ell^\infty(\Lambda;\underline{A})\times\ell^\infty(\Lambda;\underline{A})&(*=n=1)\\
      \ell^\infty(\Lambda;\underline{A})&(*=n\ge 2)\\
      0&(*\ne 0,n)
    \end{cases}
  \]
  We argue by induction on $n$. For $n=1$, note that $W_1$ is biLipschitz homeomorphic to $V_1\sqcup V_1$. Then the claim follows immediately from \eqref{(V,W)}. Suppose the claim holds for $n$. Set
  \[
    \begin{cases}
      U_+=\mathrm{Int}(D^{n+1})-\{(0,\ldots,0,t)\in\mathrm{Int}(D^{n+1})\mid t\le 0\}\\
      U_-=\mathrm{Int}(D^{n+1})-\{(0,\ldots,0,t)\in\mathrm{Int}(D^{n+1})\mid t\ge 0\}
    \end{cases}
  \]
  Then $\mathrm{Int}(D^{n+1})=(\Lambda\times U_+)\cup(\Lambda\times U_-)$, and both $\Lambda\times U_\pm$ are biLipschitz homeomorphic to $V_{n+1}$. Moreover, the inclusion $W_n\to(\Lambda\times U_+)\cap(\Lambda\times U_-)$ is an effectively proper uniform homotopy equivalence. Hence, by Theorem \ref{homotopy invariance cohomology},
  \[
    H_\mathrm{ub}^*(\Lambda\times U_+;\underline{A})\cong H_\mathrm{ub}^*(\Lambda\times U_+;\underline{A})\cong H_\mathrm{ub}^*(V_{n+1};\underline{A})
  \]
  and
  \[
    H_\mathrm{ub}^*((\Lambda\times U_+)\cap(\Lambda\times U_-);\underline{A})\cong H_\mathrm{ub}^*(W_n;\underline{A}).
  \]
  Applying the Mayer-Vietoris sequence (Theorem \ref{Mayer-Vietoris cohomology}) completes the induction.
\end{proof}

\begin{lemma}
  \label{cellular}
  Let $X$ be a uniform CW complex with index set $\Lambda$. Then, for each $n\ge 0$, there are natural isomorphisms
  \[
    H_\mathrm{ub}^*(X_n,X_{n-1};\underline{A})\cong
    \begin{cases}
      \ell^\infty(\Lambda;\underline{A})&(*=n)\\
      0&(*\ne 0)
    \end{cases}
  \]
  and
  \[
    H^\mathrm{ub}_*(X_n,X_{n-1};\underline{A})\cong
    \begin{cases}
      \ell^\infty(\Lambda;\underline{A})&(*=n)\\
      0&(*\ne 0),
    \end{cases}
  \]
  natural with respect to effectively proper uniformly continuous cellular maps between uniform CW complexes.
\end{lemma}

\begin{proof}
  We compute cohomology; the homological case is analogous. Let $\{\phi_\lambda\colon D^{n_\lambda}\to X\}_{\lambda\in\Lambda}$ be the characteristic maps of $X$, and let $U=X_n-\{\phi_\lambda(0)\}_{\lambda\in\Lambda}$. Analogously to Lemma \ref{continuity}, we verify that the inclusion $X_{n-1}\to U$ is an effectively proper uniform homotopy equivalence. Since the family $\{\phi_\lambda^{-1}\}_{\lambda \in \Lambda}$ is equicontinuous, there exists $\epsilon>0$ such that $B_\epsilon(X_{n-1})\subset U$. Hence, by Theorems \ref{cohomology exact sequence} and \ref{excision cohomology},
  \[
    H^*_\mathrm{ub}(X_n,X_{n-1};\underline{A})\cong H^*_\mathrm{ub}(X_n,U;\underline{A})\cong H^*_\mathrm{ub}(X_n-X_{n-1},U-X_{n-1};\underline{A}).
  \]
  There are biLipschitz homeomorphisms
  \[
    X_n-N_\epsilon(X_{n-1})\cong\Lambda_n\times\mathrm{Int}(D^n)\quad\text{and}\quad U-N_\epsilon(X_{n-1})\cong\Lambda_n\times\mathrm{Int}(D^n-0)
  \]
  where $\Lambda_n$ is identified with the uniformly discrete set of centers $\{\phi_\lambda(0)\}_{\lambda\in\Lambda_n}$, and the metrics are as in Lemma \ref{H(V,W)}. The result follows from Lemma \ref{H(V,W)}.
\end{proof}

\begin{definition}
  Let $X$ be a uniform CW complex. The \emph{uniform bounded cellular (co)chain complex} with coefficients in a pseudo-normed abelian group $\underline{A}$ is defined by
  \[
    \mathrm{Cell}^n_\mathrm{ub}(X;\underline{A})=H^n_\mathrm{ub}(X_n,X_{n-1};\underline{A})\quad\text{and}\quad \mathrm{Cell}_n^\mathrm{ub}(X;\underline{A})=H_n^\mathrm{ub}(X_n,X_{n-1};\underline{A})
  \]
  with (co)boundary maps
  \[
    \delta\colon\mathrm{Cell}^n_\mathrm{ub}(X;\underline{A})\to\mathrm{Cell}^{n+1}_\mathrm{ub}(X;\underline{A})\quad\text{and}\quad\partial\colon\mathrm{Cell}_{n+1}^\mathrm{ub}(X;\underline{A})\to\mathrm{Cell}_{n}^\mathrm{ub}(X;\underline{A})
  \]
  induced by the (co)boundary maps in the long exact sequence of the triple $(X_{n+1},X_{n},X_{n-1})$.
\end{definition}

By construction, uniform bounded cellular (co)homology is natural with respect to effectively proper uniformly continuous cellular maps between uniform CW complexes. Analogously to the classical case, we obtain the following.

\begin{theorem}
  Let $X$ be a uniform CW complex. Then there are natural isomorphisms
  \[
    H^*_\mathrm{ub}(X;\underline{A})\cong H^*(\mathrm{Cell}^*_\mathrm{ub}(X;\underline{A}))\quad\text{and}\quad H_*^\mathrm{ub}(X;\underline{A})\cong H_*(\mathrm{Cell}_*^\mathrm{ub}(X;\underline{A})),
  \]
  natural with respect to effectively proper uniformly continuous cellular maps between uniform CW complexes.
\end{theorem}


\section{Poincar\'e duality}\label{Poincare duality section}

In this section, we establish the Poincar\'e duality for the uniform bounded (co)homology of an oriented connected complete Riemannian manifold of bounded geometry.

We begin by defining the cap product at the level of uniform bounded (co)chains. A pseudo-normed ring $\underline{R}$ is a ring $R$ with a function $|\cdot|\colon R\to\R_{\ge 0}$ satisfying
\[
  |0|=0,\quad|a+b|\le|a|+|b|,\quad|ab|\le|a||b|
\]
for $a,b\in R$. A basic example is $\underline{\Z}$, the integers with the absolute value. Let $X$ be a metric space. Denote by $e_0,\ldots,e_q$ the standard basis of $\R^{q+1}$, and recall that $\Delta^q$ is the convex hull of these points. For $0\le i_1<\cdots<i_k\le q$, we write $[i_1,\ldots,i_k]$ for the face of $\Delta^q$ spanned by $e_{i_1},\ldots,e_{i_k}\in\R^{q+1}$. Let $u\in C_\mathrm{ub}^p(X;\underline{R})$ and $c=\sum_{i\in I}a_i\sigma_i\in C_q^\mathrm{ub}(X;\underline{R})$. We define the cap product
\[
  u\frown c=(-1)^{p(q-p)}\sum_{i\in I}a_iu(\sigma_i\vert_{[q-p,q-p+1,\ldots,q]})\sigma_i\vert_{[0,1,\ldots,q-p]}.
\]
It is straightforward to verify that $u\frown c$ is again an element of $C^\mathrm{ub}_{q-p}(X;\underline{R})$ and that the boundary operator satisfies
\begin{equation}
  \label{cap product d}
  \partial(u\frown c)=(\delta u)\frown c+(-1)^pu\frown(\partial c).
\end{equation}
Consequently, we obtain a well-defined pairing
\[
  \frown\colon H^p_\mathrm{ub}(X;\underline{R})\otimes H_q^\mathrm{ub}(X;\underline{R})\to H_{q-p}^\mathrm{ub}(X;\underline{R}),\quad[u]\otimes[c]\mapsto[u]\frown[c]=[u\frown c].
\]
For a subspace $A\subset X$, the cap product extends naturally to
\[
  \frown\colon H^p_\mathrm{ub}(X;\underline{R})\otimes H_q^\mathrm{ub}(X,A;\underline{R})\to H_{q-p}^\mathrm{ub}(X,A;\underline{R})
\]
and, clearly, it restricts to the $\epsilon$-controlled complexes as
\[
  \frown\colon H^p_{\mathrm{ub},\epsilon}(X;\underline{R})\otimes H_q^{\mathrm{ub},\epsilon}(X,A;\underline{R})\to H_{q-p}^{\mathrm{ub},\epsilon}(X,A;\underline{R}).
\]

By Lemma \ref{H(V,W)}, we have
\[
  H^*_\mathrm{ub}(\mathrm{Int}(D^n))\cong
  \begin{cases}
    \Z&(*=0)\\
    0&(*\ne 0)
  \end{cases}
\]
and
\[
  H_*^\mathrm{ub}(\mathrm{Int}(D^n),\mathrm{Int}(D^n-0))\cong
  \begin{cases}
    \Z&(*=n)\\
    0&(*\ne n).
  \end{cases}
\]
Moreover, the proof of Lemma \ref{H(V,W)} shows that $H_n^\mathrm{ub}(\mathrm{Int}(D^n),\mathrm{Int}(D^n-0))\cong\Z$ is generated by the class represented by the composite
\[
  \Delta^n\cong D_\frac{1}{2}^n\xrightarrow{\mathrm{incl}}\mathrm{Int}(D^n)
\]
where $D_\frac{1}{2}^n\subset D^n$ denotes the $n$-disk of radius $\frac{1}{2}$. We denote this class by $[\mathrm{Int}(D^n)]$. For $V_n$ in Lemma \ref{H(V,W)}, define
\[
  [V_n]=\left[\sum_{\lambda\in\Lambda}(\lambda,\mathrm{Int}(D_\lambda^n))\right].
\]

\begin{lemma}
  \label{local Poincare duality}
  Let $V_n$ and $W_n$ be as in Lemma \ref{H(V,W)}. The map
  \[
    H^p_\mathrm{ub}(\Lambda\times\mathrm{Int}(D^n))\to H_{n-p}^\mathrm{ub}(\Lambda\times(\mathrm{Int}(D^n),\mathrm{Int}(D^n-0))),\quad u\mapsto u\frown[V_n]
  \]
  is an isomorphism for all $p\ge 0$.
\end{lemma}

\begin{proof}
  It is classical that the map
  \[
    H^p_\mathrm{ub}(\mathrm{Int}(D^n))\to H_{n-p}^\mathrm{ub}(\mathrm{Int}(D^n),\mathrm{Int}(D^n-0)),\quad u\mapsto u\frown[\mathrm{Int}(D^n)]
  \]
  is an isomorphism for all $p\ge 0$. Therefore, the claim follows from Lemma \ref{H(V,W)} together with \eqref{V_n}.
\end{proof}

For the rest of this section, let $M$ be an oriented connected complete Riemannian $n$-dimensional manifold of bounded geometry. By Example \ref{triangulation}, $M$ is biLipschitz homeomorphic to a uniform simplicial complex. Hence, when dealing with uniform bounded (co)homology, we may identify $M$ with its triangulating uniform simplicial complex. Fix a sufficiently small $\epsilon>0$, and for a subset $A\subset M$, let $N_\epsilon(A)$ denote the $\epsilon$-neighborhood of $A$. Fix an ordering $\sigma_1<\sigma_2<\cdots$ of $n$-simplices of $M$. For $\alpha_0,\ldots,\alpha_p\in\N$, set
\[
  \sigma_{\alpha_0\cdots\alpha_p}=
  \begin{cases}
    \sigma_{\alpha_0}\cap\cdots\cap\sigma_{\alpha_p}&(\dim\sigma_{\alpha_0}\cap\cdots\cap\sigma_{\alpha_p}=n-p)\\
    \emptyset&(\text{otherwise}).
  \end{cases}
\]
For each $p\ge 0$, define
\[
  K_{p,q}=C_{\mathrm{ub},\epsilon}^q\left(\coprod_{\alpha_0<\cdots<\alpha_p}N_\epsilon(\sigma_{\alpha_0\cdots\alpha_p})\right).
\]
For a cochain $u$, let $u_{\alpha_0\cdots\alpha_p}$ denote its restriction to $N_\epsilon(\sigma_{\alpha_0\cdots\alpha_p})$, and impose
\[
  u_{\alpha_0\cdots\alpha_i\cdots\alpha_j\cdots\alpha_p}=-u_{\alpha_0\cdots\alpha_j\cdots\alpha_i\cdots\alpha_p}.
\]
Thus $K_{p,q}\subset\prod_{\alpha_0,\ldots,\alpha_p}C_{\mathrm{ub},\epsilon}^q(N_\epsilon(\sigma_{\alpha_0\cdots\alpha_p}))$, and every $u\in K_{p,q}$ is completely determined by the collection $\{u_{\alpha_0\cdots\alpha_p}\}_{\alpha_0,\ldots,\alpha_p}$.

Define $d\colon K_{p,q}\to K_{p+1,q}$ by setting, for $u\in K_{p+1,q}$,
\[
  (du)_{\alpha_0\cdots\alpha_{p+1}}=\left(\sum_{i=0}^{p+1}(-1)^iu_{\alpha_0\cdots\hat{\alpha}_i\cdots\alpha_{p+1}}\right)_{\alpha_0\cdots\alpha_{p+1}}
\]
and let $\delta\colon K_{p,q}\to K_{p,q+1}$ be the coboundary operator of uniform bounded cochain complexes.

\begin{lemma}
  \label{double complex 1}
  The triple $(K,d,\delta)$ forms a double complex.
\end{lemma}

\begin{proof}
  For $u\in K_{p,q}$,
  \begin{align*}
    (d^2u)_{\alpha_0\cdots\alpha_{p+2}}&=\left(\sum_{i=0}^{p+2}(-1)^i(du)_{\alpha_0\cdots\hat{\alpha}_i\cdots\alpha_{p+2}}\right)_{\alpha_0\cdots\alpha_{p+2}}\\
    &=\left(\sum_{i=0}^{p+2}\sum_{j<i}(-1)^{i+j}u_{\alpha_0\cdots\hat{\alpha}_j\cdots\hat{\alpha}_i\cdots\alpha_{p+2}}+\sum_{i<j}(-1)^{i+j-1}u_{\alpha_0\cdots\hat{\alpha}_i\cdots\hat{\alpha}_j\cdots\alpha_{p+2}}\right)_{\alpha_0\cdots\alpha_{p+2}}\\
    &=0.
  \end{align*}
  Since the restriction of cochains commutes with $\delta$, we have $d\delta=\delta d$. Finally, $\delta^2=0$. Hence $(K,d,\delta)$ is a double complex.
\end{proof}

The restriction maps $C_{\mathrm{ub},\epsilon}^*(M)\to C^*_{\mathrm{ub},\epsilon}(N_\epsilon(\sigma_\alpha))$ for $\alpha\ge 1$ assemble into a cochain map
\[
  \theta\colon C_{\mathrm{ub},\epsilon}^*(M)\to K_{0,*}
\]

\begin{lemma}
  \label{exact 1}
  The sequence of cochain complexes
  \[
    0\to  C_{\mathrm{ub},\epsilon}^*(M)\xrightarrow{\theta}K_{0,*}\xrightarrow{d}K_{1,*}\xrightarrow{d}\cdots
  \]
  is exact.
\end{lemma}

\begin{proof}
  The exactness of $0\to  C_\epsilon^q(M)\xrightarrow{\theta}K_{0,q}\xrightarrow{d}K_{1,q}$ is clear. For $u\in C^q_{\mathrm{ub,\epsilon}}(N_\epsilon(\sigma_{\alpha_0\cdots\alpha_p}))$ and each $\alpha\in\N$, define $\rho_\alpha(u)\in C^q_{\mathrm{ub,\epsilon}}(M)$ by
  \[
    \rho_\alpha(u)(\tau)=
    \begin{cases}
      u(\tau)&(\tau\subset N_\epsilon(\sigma_\alpha)\cap N_\epsilon(\sigma_{\alpha_0\cdots\alpha_p})\quad\text{and}\quad\tau\not\subset N_\epsilon(\sigma_1)\cup\cdots\cup N_\epsilon(\sigma_{\alpha-1}))\\
      0&(\text{otherwise})
    \end{cases}
  \]
  for every singular $q$-simplex $\tau$ in $M$ with $\diam(\tau)<\epsilon$. For every singular simplex $\tau$ in $M$ with $\diam(\tau)<\epsilon$, there exists $\alpha\in\N$ such that $\tau\subset N_\epsilon(\sigma_\alpha)$. Hence there is a unique $\alpha\in\N$ such that $\tau\subset N_\epsilon(\sigma_\alpha)$ and $\tau\not\subset N_\epsilon(\sigma_1)\cup\cdots\cup N_\epsilon(\sigma_{\alpha-1})$. Consequently,
  \[
    \left(\sum_\alpha\rho_\alpha(u)\right)_{\alpha_0\cdots\alpha_p}=u.
  \]
  For $p\ge 1$, let $u\in K_{p,q}$. Define a map $H\colon K_{p,q}\to K_{p-1,q}$ by
  \[
    (Hu)_{\alpha_0\cdots\alpha_{p-1}}=\left(\sum_\alpha\rho_\alpha(u_{\alpha\alpha_0\cdots\alpha_{p-1}})\right)_{\alpha_0\cdots\alpha_{p-1}}.
  \]
  Then we compute
  \begin{align*}
    (dHu)_{\alpha_0\cdots\alpha_p}&=\left(\sum_{i=0}^p(-1)^i(Hu)_{\alpha_0\cdots\hat{\alpha}_i\cdots\alpha_p}\right)_{\alpha_0\cdots\alpha_p}\\
    &=\left(\sum_{i=0}^p(-1)^i\sum_\alpha\rho_\alpha(u_{\alpha\alpha_0\cdots\hat{\alpha}_i\cdots\alpha_p})\right)_{\alpha_0\cdots\alpha_p}
  \end{align*}
  and
  \begin{align*}
    (Hdu)_{\alpha_0\cdots\alpha_p}&=\left(\sum_\alpha\rho_\alpha((du)_{\alpha\alpha_0\cdots\alpha_p})\right)_{\alpha_0\cdots\alpha_p}\\
    &=\left(\sum_\alpha\rho_\alpha(u_{\alpha_0\cdots\alpha_p})+\sum_{i=0}^p(-1)^{i+1}\rho_\alpha(u_{\alpha\alpha_0\cdots\hat{\alpha}_i\cdots\alpha_p})\right)_{\alpha_0\cdots\alpha_p}\\
    &=u_{\alpha_0\cdots\alpha_p}-(dHu)_{\alpha_0\cdots\alpha_p}.
  \end{align*}
  Therefore $dH+Hd=1$, completing the proof.
\end{proof}

For $\alpha_0,\ldots,\alpha_p\in\N$, set
\[
  L(\sigma_{\alpha_0\cdots\alpha_p})=(N_\epsilon(\sigma_{\alpha_0\cdots\alpha_p}),N_\epsilon(\sigma_{\alpha_0\cdots\alpha_p})-\sigma_{\alpha_0\cdots\alpha_p})
\]
and define
\[
  L_{p,q}=C_q^{\mathrm{ub},\epsilon}\left(\coprod_{\alpha_0<\cdots<\alpha_p}L(\sigma_{\alpha_0\cdots\alpha_p})\right).
\]
For $c\in L_{p,q}$ and $\alpha_0,\ldots,\alpha_p\in\N$, let $c_{\alpha_0\cdots\alpha_p}$ denote the restriction of $c$ to $C_q^{\mathrm{ub},\epsilon}(L(\sigma_{\alpha_0\cdots\alpha_p}))$, and set
\[
  c_{\alpha_0\cdots\alpha_i\cdots\alpha_j\cdots\alpha_p}=-c_{\alpha_0\cdots\alpha_j\cdots\alpha_i\cdots\alpha_p}.
\]
Define $d\colon L_{p,q}\to L_{p+1,q}$ by
\[
  (dc)_{\alpha_0\cdots\alpha_{p+1}}=\sum_{i=0}^{p+1}(-1)^ic_{\alpha_0\cdots\hat{\alpha}_i\cdots\alpha_{p+1}},
\]
interpreting $c_{\alpha_0\cdots\hat{\alpha}_i\cdots\alpha_{p+1}}$ in $C_q^{\mathrm{ub},\epsilon}(L(\sigma_{\alpha_0\cdots\alpha_{p+1}}))$ through the composite
\begin{align*}
  C_q^{\mathrm{ub},\epsilon}(L(\sigma_{\alpha_0\cdots\hat{\alpha}_i\cdots\alpha_{p+1}}))&\to C_q^{\mathrm{ub},\epsilon}(N_\epsilon(\sigma_{\alpha_0\cdots\hat{\alpha}_i\cdots\alpha_{p+1}}),N_\epsilon(\sigma_{\alpha_0\cdots\hat{\alpha}_i\cdots\alpha_{p+1}})-\sigma_{\alpha_0\cdots\alpha_{p+1}})\\
  &\cong C_q^{\mathrm{ub},\epsilon}(L(\sigma_{\alpha_0\cdots\alpha_{p+1}})).
\end{align*}
Let $\partial\colon L_{p,q}\to L_{p,q-1}$ be the boundary operator. Exactly as in Lemma \ref{double complex 1}, we have:

\begin{lemma}
  \label{double complex 2}
  The triad $(L,d,\partial)$ is a double complex.
\end{lemma}

The restriction maps
\[
  C^{\mathrm{ub},\epsilon}_*(M)\to H_*^{\mathrm{ub},\epsilon}(N_\epsilon(\sigma_\alpha),N_\epsilon(\sigma_\alpha)-\sigma_\alpha)
\]
assemble into a chain map $C^{\mathrm{ub},\epsilon}_*(M)\to L_{0,*}$. Arguing as in Lemma \ref{exact 1}, we obtain:

\begin{lemma}
  \label{exact 2}
  The sequence of chain complexes
  \[
    0\to  C^{\mathrm{ub},\epsilon}_*(M)\xrightarrow{\theta}L_{0,*}\xrightarrow{d}L_{1,*}\xrightarrow{d}\cdots
  \]
  is exact.
\end{lemma}

Since $M$ is oriented, each $n$-simplex of (the triangulating uniform simplicial complex of) $M$ is oriented. Let $\mu$ denote the (possibly infinite) sum of all oriented $n$-simplices. Then $\mu$ is a uniform bounded chain in $M$. Every $(n-1)$-simplex is contained in exactly two $n$-simplices with opposite induced orientations, hence $\mu$ is a cycle. Its homology class $[M]\in H_n^\mathrm{ub}(M)$ is called the \emph{fundamental class} of $M$. By the uniformity of the triangulating simplicial complex, after finitely many barycentric subdivisions, we may assume that every $n$-simplex has diameter $<\epsilon$. Let $\mu_\epsilon$ denote the sum of all oriented $n$-simplices of this subdivision. Then $\mu_\epsilon$ represents the same class $[M]$. Define
\[
  \mathrm{PD}\colon K_{p,q}\to L_{p,n-q},\quad(\mathrm{PD}(u))_{\alpha_0\cdots\alpha_p}=u_{\alpha_0\cdots\alpha_p}\frown(\mu_\epsilon)_{\alpha_0\cdots\alpha_p}.
\]
By \eqref{cap product d},
\[
  \mathrm{PD}\circ\delta=\partial\circ\mathrm{PD},
\]
so $\mathrm{PD}\colon K_{p,*}\to L_{p,n-*}$ is a map of (co)chain complexes (up to degree shift).

\begin{lemma}
  \label{cap}
  There is a commutative diagram
  \[
    \xymatrix{
      0\ar[r]&C_{\mathrm{ub},\epsilon}^q(M)\ar[d]^{\frown\mu}\ar[r]^\theta&K_{0,q}\ar[d]^{\mathrm{PD}}\ar[r]^d&K_{1,q}\ar[r]^d\ar[d]^{\mathrm{PD}}&\cdots\\
      0\ar[r]&C^{\mathrm{ub},\epsilon}_{n-q}(M)\ar[r]^\theta&L_{0,n-q}\ar[r]^d&L_{1,n-q}\ar[r]^d&\cdots.
    }
  \]
\end{lemma}

\begin{proof}
  The commutativity of the leftmost square is immediate from the definition of $\mathrm{PD}\colon K_{0,q}\to L_{0,n-q}$. For $u\in K_{p,q}$,
  \begin{align*}
    (d\circ\mathrm{PD}(u))_{\alpha_0\cdots\alpha_{p+1}}&=\left(\sum_{i=0}^{p+1}(-1)^i\mathrm{PD}(u)_{\alpha_0\cdots\hat{\alpha}_i\cdots\alpha_{p+1}}\right)_{\alpha_0\cdots\alpha_{p+1}}\\
    &=\left(\sum_{i=0}^{p+1}(-1)^iu_{\alpha_0\cdots\hat{\alpha}_i\cdots\alpha_{p+1}}\frown(\mu_\epsilon)_{\alpha_0\cdots\hat{\alpha}_i\cdots\alpha_{p+1}}\right)_{\alpha_0\cdots\alpha_{p+1}}\\
    &=\left(\sum_{i=0}^{p+1}(-1)^iu_{\alpha_0\cdots\hat{\alpha}_i\cdots\alpha_{p+1}}\right)_{\alpha_0\cdots\alpha_{p+1}}\frown(\mu_\epsilon)_{\alpha_0\cdots\alpha_{p+1}}\\
    &=(du)_{\alpha_0\cdots\alpha_{p+1}}\frown(\mu_\epsilon)_{\alpha_0\cdots\alpha_{p+1}}\\
    &=(\mathrm{PD}\circ d(u))_{\alpha_0\cdots\alpha_{p+1}},
  \end{align*}
  giving desired commutativity.
\end{proof}

We now obtain the Poincar\'e duality in uniform bounded (co)homology.

\begin{theorem}
  \label{Poincare duality}
  Let $M$ be an oriented connected complete
  Riemannian $n$-manifold of bounded geometry. The map
  \[
    \frown[M]\colon H^q_\mathrm{ub}(M)\to H_{n-q}^\mathrm{ub}(M)
  \]
  is an isomorphism for all $q\ge 0$.
\end{theorem}

\begin{proof}
    As in \cite[p.141]{Wb}, there is a spectral sequence
    \[
      E^1=H_*(K;d)\quad\Longrightarrow\quad H_*(\mathrm{Tot}(K)).
    \]
    where $\mathrm{Tot}(K)$ denotes the total complex of the double complex and $d^1\colon E^1_{p,q}\to E^1_{p,q+1}$ is induced by $\delta$. By Lemma \ref{exact 1},
    \[
      E^1_{p,q}=
      \begin{cases}
        C^q_{\mathrm{ub},\epsilon}(M)&(p=0)\\
        0&(p\ne 0).
      \end{cases}
    \]
    Hence $H_*(\mathrm{Tot}(K))\cong H^*_{\mathrm{ub},\epsilon}(M)$. By symmetry, we also have a spectral sequence
    \[
      {}_KE^1=H_*(K;\delta)\quad\Longrightarrow\quad H^*_{\mathrm{ub},\epsilon}(M)
    \]
    and similarly,
    \[
      {}_LE^1=H_*(L;\partial)\quad\Longrightarrow\quad H_*^{\mathrm{ub},\epsilon}(M).
    \]
    By Lemma \ref{cap}, there is a map of spectral sequences $\mathrm{PD}\colon{}_KE^r_{p,q}\to{}_LE^r_{p,n-q}$. By Lemma \ref{local Poincare duality}, this map is an isomorphism for $r=1$, and hence also for $r=\infty$. Consequently, the map
    \[
      H_{\mathrm{ub},\epsilon}^q(M)\to H^{\mathrm{ub},\epsilon}_{n-q}(M),\quad x\mapsto x\frown[\mu_\epsilon]
    \]
    is an isomorphism. As $\mu_\epsilon$ represents the fundamental class $[M]$, the result follows from Lemma \ref{C_epsilon}.
\end{proof}


\section{Obstruction theory}\label{Obstruction theory}

This section generalizes classical obstruction theory (for deforming sections) to uniformly continuous settings. That is, instead of just working with CW complexes and topological fiber bundles, we now require uniform control - all maps, homotopies, and local trivializations are uniformly continuous and uniformly equicontinuous families.

\subsection{Uniform relative fiber bundle}

A \emph{relative fiber bundle} is a pair of fiber bundles $p\colon E\to B$ and $p\vert_{E_0}\colon E_0\to B$, where $E_0$ is a subspace of $E$. If the fibers of $p\colon E\to B$ and $p\vert_{E_0}\colon E_0\to B$ are $F$ and $F_0$, respectively, then we denote the relative fiber bundle by
\[
  (F,F_0)\to(E,E_0)\xrightarrow{p}B.
\]

\begin{definition}
  \label{uniform fiber bundle}
  A relative fiber bundle $(F,F_0)\to(E,E_0)\xrightarrow{p}B$ of metric spaces is said to be \emph{uniform} if the following conditions hold:

  \begin{enumerate}
    \item The projection $p\colon E\to B$ is uniformly continuous.

    \item There is a cover $\{U_i\}_{i\in I}$ of $B$ with positive Lebesgue number.

    \item For each  $i\in I$, there is a fiberwise homeomorphism
    \[
      h_i\colon(p^{-1}(U_i),p^{-1}(U_i)\cap E_0)\xrightarrow{\cong}U_i\times(F,F_0)
    \]
    such that both families $\{h_i\}_{i\in I}$ and $\{h_i^{-1}\}_{i\in I}$ are uniformly equicontinuous.
  \end{enumerate}
\end{definition}

For the remainder of this section, let $(F,F_0)\to(E,E_0)\xrightarrow{p}B$ be a uniform relative fiber bundle above.

\begin{lemma}
  \label{pullback bundle}
  Let $f\colon X\to B$ be a uniformly continuous map. Then the induced relative fiber bundle
  \[
    (F,F_0)\to(f^{-1}E,f^{-1}E_0)\xrightarrow{q}X
  \]
  is uniform.
\end{lemma}

\begin{proof}
  Since $f^{-1}E=X\times_BE$, the projection $q\colon f^{-1}E\to X$ is uniformly continuous. Let $V_i=f^{-1}(U_i)$ for $i\in I$. Suppose that the cover $\{U_i\}_{i\in I}$ has a Lebesgue number $\epsilon>0$. Because $f$ is uniformly continuous, there exists $\delta>0$ such that for any subset $A\subset X$ with diameter $<\delta$, $f(A)$ has diameter $<\epsilon$. Hence $f(A)\subset U_i$ for some $i\in I$, and thus $A\subset V_i$. Therefore $\{V_i\}_{i\in I}$ forms a cover of $X$ with Lebesgue number $\delta>0$. Observe that $(q^{-1}(V_i),q^{-1}(V_i)\cap f^{-1}E_0)=V_i\times_{U_i}(p^{-1}(U_i),p^{-1}(U_i)\cap E_0)$ and $V_i\times(F,F_0)=(V_i\times_{U_i}U_i)\times(F,F_0)$. The restriction of the map
  \[
    1\times h_i\colon V_i\times(p^{-1}(U_i),p^{-1}(U_i)\cap E_0)\to V_i\times U_i\times(F,F_0)
  \]
  gives a fiberwise homeomorphism
  \[
    g_i\colon(q^{-1}(V_i),q^{-1}(V_i)\cap f^{-1}E_0)\xrightarrow{\cong}V_i\times(F,F_0).
  \]
  Since both $\{h_i\}_{i\in I}$ and $\{h_i^{-1}\}_{i\in I}$ are uniformly equicontinuous, so are$\{g_i\}_{i\in I}$ and $\{g_i^{-1}\}_{i\in I}$. This completes the proof.
\end{proof}

\begin{definition}
  A section $s\colon B\to E$ of $p\colon E\to B$ is \emph{bounded} if the union of the images of the compositions
  \[
    U_i\xrightarrow{s}p^{-1}(U_i)\xrightarrow{h_i}U_i\times F\xrightarrow{\mathrm{proj}}F
  \]
  for all $i\in I$ is a bounded subset of $F$.
\end{definition}

Recall that two sections $s_0,s_1\colon B\to E$ of $p\colon E\to B$ are \emph{fiberwise homotopic} if there is a map $H\colon B\times[0,1]\to E$ such that $H(-,0)=s_0$, $H(-,1)=s_1$ and for each $t\in[0,1]$, the map $H(-,t)\colon B\to E$ is a section of $p\colon E\to B$.

\begin{lemma}
  Let $s_0,s_1\colon B\to E$ be sections of $p\colon E\to B$. If $s_0$ is bounded and is uniformly fiberwise homotopic to $s_1$, then $s_1$ is also bounded.
\end{lemma}

\begin{proof}
  Let $H\colon B\times[0,1]\to E$ be a uniform fiberwise homotopy between $s_0$ and $s_1$. Since $H$ is uniform, by an argument analogous to that of Lemma \ref{EPUC homotopy}, there exists $R>0$ such that
  \[
    d(s_0(x),s_1(x))=d(H(x,0),H(x,1))<R
  \]
  for all $x\in B$. Hence $s_1$ is bounded.
\end{proof}

We consider the following two conditions to build up obstruction theory for a uniform relative fiber bundle.

\begin{definition}
  \label{admissible}
  A pair of metric spaces $(X,A)$ is \emph{$n$-admissible} for $n\ge 2$ if:

  \begin{enumerate}
    \item $A$ is path-connected.

    \item $\pi_n(X,A)$ is a finitely generated abelian group, and the action of $\pi_1(A)$ on $\pi_n(X,A)$ is trivial.

    \item For every $\epsilon>0$, there exists $\delta>0$ such that if equicontinuous families of maps $\{f_i\colon D^n\to X\}_{i\in I}$ and $\{g_i\colon D^n\to X\}$ satisfy $d(f_i(x),g_i(x))<\delta$ for all $x\in D^n$ and $i\in I$, then for each $i\in I$, there is a homotopy $H_i\colon D^n\times[0,1]\to X$ from $f_i$ to $g_i$ such that $\{H_i\}_{i\in I}$ is equicontinuous and
    \[
      H_i(x\times[0,1])\subset N_\epsilon(f_i(x))
    \]
    for all $x\in D^n$ and $i\in I$.
  \end{enumerate}
\end{definition}

\begin{remark}
  \label{admissible condition}
  Condition (3) is satisfied whenever $X$ is a complete Riemannian manifold whose injectivity radius is positive.
\end{remark}



\begin{definition}
  \label{good}
  A uniform relative fiber bundle $(F,F_0)\to(E,E_0)\xrightarrow{p}B$ is called \emph{good} if:

   \begin{enumerate}
     \item $E_0$ is open in $E$.

     \item $F$ is complete and proper, and $F_0$ is nonempty and path-connected.

     \item $B$ is a uniform CW complex such that each closed cell is contained in some $U_i$.
   \end{enumerate}
\end{definition}

\subsection{Obstruction class}

For a subset $A$ of a metric space $X$ and $\epsilon>0$, define
\[
  \mathrm{Int}_\epsilon(A)=\{x\in A\mid N_\epsilon(x)\subset A\}.
\]
We state the main theorem of this section.

\begin{theorem}
  \label{obstruction}
  Let $(F,F_0)\to(E,E_0)\xrightarrow{p}B$ be a good uniform relative fiber bundle satisfying the following conditions:

  \begin{enumerate}
    \item $(F,F_0)$ is $n$-admissible.

    \item For any $i,j\in I$, the map
    \[
      h_j\circ h_i^{-1}\colon(U_i\cap U_j)\times(F,F_0)\to(U_i\cap U_j)\times(F,F_0)
    \]
    induces the identity map on $\pi_n(F,F_0)$ at every $x\in U_i\cap U_j$.
  \end{enumerate}

  \noindent Suppose there is a bounded uniformly continuous section $s_{n-1}\colon B\to E$ of $p\colon E\to B$ satisfying $s_{n-1}(B_{n-1})\subset\mathrm{Int}_{\epsilon_{n-1}}(E_0)$ for some $\epsilon_{n-1}>0$. Then there exists an obstruction class
  \[
    \mathfrak{o}_n(E,E_0;s_{n-1})\in H_\mathrm{ub}^n(B;\underline{\pi_n(F,F_0)})
  \]
  such that the following are equivalent:

  \begin{enumerate}
    \item[(i)] $\mathfrak{o}_n(E,E_0;s_{n-1})=0$.

    \item[(ii)] The section $s_{n-1}$ is uniformly fiberwise homotopic to a bounded uniformly continuous section $s_n\colon B\to E$ satisfying
    \[
      s_n(B_n)\subset\mathrm{Int}_{\epsilon_n}(E_0)
    \]
    for some $\epsilon_n>0$, through a uniform fiberwise homotopy $H\colon B\times[0,1]\to E$ such that $H(B_{n-2}\times[0,1])\subset\mathrm{Int}_{\epsilon_n}(E_0)$.
  \end{enumerate}

  \noindent Moreover, for an effectively proper uniformly continuous cellular map $f\colon X\to B$ between uniform CW complexes, we have
  \[
    f^*(\mathfrak{o}_n(E,E_0;s))=\mathfrak{o}_n(f^{-1}E,f^{-1}E_0;f^{-1}s).
  \]
\end{theorem}

For the remainder of this section, we assume that $(F,F_0)\to(E,E_0)\xrightarrow{p}B$ satisfies the hypotheses of Theorem \ref{obstruction}.

\begin{lemma}
  \label{pi_n}
  The action of $\pi_1(E_0)$ on $\pi_n(E,E_0)$ is trivial.
\end{lemma}

\begin{proof}
  Since $F_0$ is nonempty, we have the homotopy exact sequence
  \[
    \cdots\to\pi_{n+1}(B,B)\to\pi_n(F,F_0)\to\pi_n(E,E_0)\to\pi_n(B,B)\to\cdots.
  \]
  Because $\pi_*(B,B)=0$, the inclusion $(F,F_0)\to(E,E_0)$ is an isomorphism in $\pi_n$. Consider commutative diagram
  \[
    \xymatrix{
      \pi_1(F_0)\times\pi_n(F,F_0)\ar[r]\ar[d]&\pi_n(F,F_0)\ar[d]\\
      \pi_1(E_0)\times\pi_n(E,E_0)\ar[r]&\pi_n(E,E_0),
    }
  \]
  where the horizontal maps are the action of $\pi_1$. Since the action of $\pi_1(F_0)$ on $\pi_n(F,F_0)$ is trivial, the action of $\pi_1(E_0)$ on $\pi_n(E,E_0)$ factors through that of $\pi_1(B)$ on $\pi_n(E,E_0)\cong\pi_n(F,F_0)$, which is trivial by condition (2) of Theorem \ref{obstruction}. Hence the claim follows.
\end{proof}

Let $s_{n-1}\colon B\to E$ be a section satisfying $s_{n-1}(B_{n-1})\subset\mathrm{Int}_{\epsilon_{n-1}}(E_0)$ for some $\epsilon_{n-1}>0$. For pairs of spaces $(X,Y)$ and $(Z,W)$, denote by $\pi_0(X,Y;Z,W)$ the set of free homotopy classes of maps $(X,Y)\to(Z,W)$. Let $\Lambda$ be the set of cells of $B$. By Lemma \ref{pi_n}, the composite of the natural maps
\[
  \pi_n(F,F_0)\to\pi_n(E,E_0)\to\pi_0(D^n,S^{n-1};E,E_0)
\]
is an isomorphism. Hence, for each $\lambda\in\Lambda_n$, we define
\[
  o_n(E,E_0;s_{n-1})\colon\Lambda_n\to\pi_0(D^n,S^{n-1};E,E_0)\cong\pi_n(F,F_0),\quad\lambda\mapsto[s_{n-1}\circ\phi_\lambda],
\]
where $\{\phi_\lambda\colon D^{n_\lambda}\to B\}_{\lambda\in\Lambda}$ are characteristic maps of $B$. For each $\lambda\in\Lambda_n$, choose $i_\lambda\in I$ such that $\phi_\lambda(D^n)\subset U_{i_\lambda}$, and define
\begin{multline*}
  \Phi_\lambda\colon(D^n,S^{n-1})\xrightarrow{\phi_\lambda}(U_{i_\lambda},U_{i_\lambda}\cap B_{n-1})\xrightarrow{s}(p^{-1}(U_{i_\lambda}),p^{-1}(U_{i_\lambda})\cap E_0)\\
  \xrightarrow{h_{i_\lambda}}U_{i_\lambda}\times(F,F_0)\xrightarrow{\mathrm{proj}}(F,F_0).
\end{multline*}

\begin{lemma}
  \label{Ascoli-Arzela}
  There is a finite subset $\{\mu_1,\ldots,\mu_k\}\subset\Lambda_n$ satisfying:

  \begin{enumerate}
    \item For every $\lambda\in\Lambda_n$, there exist $\mu(\lambda)\in\{\mu_1,\ldots,\mu_k\}$ and a homotopy $H_\lambda\colon(D^n,S^{n-1})\times[0,1]\to(F,F_0)$ from $\Phi_\lambda$ to $\Phi_{\mu(\lambda)}$.

    \item There exists $\epsilon>0$ such that $H_\lambda(S^{n-1}\times[0,1])\subset\mathrm{Int}_\epsilon(F_0)$ for all $\lambda\in\Lambda_n$.

    \item The family $\{H_\lambda\}_{\lambda\in\Lambda}$ is equicontinuous.
  \end{enumerate}
\end{lemma}

\begin{proof}
  Consider the space of maps $D^n\to F$ with the supremum metric, and let $S$ denote the closure of $\{\Phi_\lambda\}_{\lambda\in\Lambda_n}$ in this space. Since the section $s\colon B\to E$ is bounded, $S$ is bounded. For each $x\in D^n$, the set $S_x=\{\phi(x)\mid\phi\in S\}\subset F$ is bounded, and because $F$ is proper, $S_x$ is relatively compact. Moreover, as the family of trivialization $\{h_i\}_{i\in I}$ is uniformly equicontinuous, so is $S$. By the Ascoli-Arzel\`{a} theorem \cite[21 Ascoli Theorem]{K}, there is a finite subset $\{\mu_1,\ldots,\mu_k\}\subset\Lambda_n$ such that, for any any $\delta>0$ and $\lambda\in\Lambda_n$, one can choose $\mu(\lambda)\in\{\mu_1,\ldots,\mu_k\}$ with
  \[
    d(\Phi_\lambda(x),\Phi_{\mu(\lambda)}(x))<\delta
  \]
  for all $x\in D^n$. Then by condition (3) of Definition \ref{admissible}, for each $\lambda\in\Lambda_n$, there is a homotopy $H_\lambda\colon D^n\times[0,1]\to F$ from $\Phi_\lambda$ to $\Phi_{\mu(\lambda)}$, such that the family $\{H_\lambda\}_{\lambda\in\Lambda}$ is equicontinuous. Since $E_0$ is open in $E$, $F_0$ is open in $F$. Because $D^n$ is compact, there exists $\epsilon>0$ such that
  \[
    \Phi_{\mu_1}(D^n),\ldots,\Phi_{\mu_k}(D^n)\subset\mathrm{Int}_\epsilon(F_0).
  \]
  By choosing $\delta$ sufficiently small, we ensure that $H_\lambda(S^{n-1}\times[0,1])\subset\mathrm{Int}_\epsilon(F_0)$ for all $\lambda\in\Lambda_n$. This completes the proof.
\end{proof}

\begin{lemma}
  The map
  \[
    o_n(E,E_0;s_{n-1})\colon\Lambda_n\to\pi_n(F,F_0)
  \]
  defines an element of $\mathrm{Cell}_\mathrm{ub}^n(B;\underline{\pi_n(F,F_0)})$.
\end{lemma}

\begin{proof}
  By Lemma \ref{cellular}, since $\pi_n(F,F_0)$ is finitely generated, the claim is equivalent to showing that $o_n(E,E_0;s_{n-1})(\Lambda_n)$ is a finite subset of $\pi_n(F,F_0)$. This follows immediately from Lemma \ref{Ascoli-Arzela}.
\end{proof}

Hereafter, we regard $o_n(E,E_0;s_{n-1})$ as an element of $\mathrm{Cell}_\mathrm{ub}^n(B;\underline{\pi_n(F,F_0)})$.

\begin{lemma}
  \label{cocycle}
  $\delta o_n(E,E_0;s_{n-1})=0$.
\end{lemma}

\begin{proof}
  If $B$ has no $(n+1)$-cells, then the statement is immediate. Assume that $B$ has at least one $(n+1)$-cell, and let $e$ be such a cell. Then there is a commutative diagram
  \[
    \xymatrix{
      H_{n+1}(\bar{e},\partial\bar{e})\ar[d]&\pi_{n+1}(\bar{e},\partial\bar{e})\ar[l]_\cong\ar[r]^{(s_{n-1})_*}\ar[d]&\pi_{n+1}(E,E)\ar[d]\\
      H_n(\partial\bar{e})\ar[d]&\pi_n(\partial\bar{e})\ar[l]_\cong\ar[r]^{(s_{n-1})_*}\ar[d]&\pi_n(E)\ar[d]\\
      H_n(B_n,B_{n-1})&\pi_0(D^n,S^{n-1};B_n,B_{n-1})\ar[l]\ar[r]^{(s_{n-1})_*}&\pi_0(D^n,S^{n-1};E,E_0).
    }
  \]
  Since $F_0$ is nonempty and path-connected, $B$ is path-connected as well. Hence, the bottom-left map $\pi_0(D^n,S^{n-1};B_n,B_{n-1})\to H_n(B_n,B_{n-1})$ is bijective and therefore invertible. Using this inverse, we may interpret the composite of the bottom row maps as the obstruction element $o(E,E_0;s_{n-1})$. Now consider the counterclockwise perimeter in the diagram beginning at $H_{n+1}(\bar{e},\partial\bar{e})$. This computes $(\delta o_n(E,E_0;s_{n-1}))(e)$. On the other hand, since $\pi_{n+1}(E,E)=0$, the clockwise perimeter from $H_{n+1}(\bar{e},\partial\bar{e})$ is trivial. Therefore, $(\delta o_n(E,E_0;s_{n-1}))(e)=0$. As this holds for any $(n+1)$-cell $e$, we conclude that $\delta o_n(E,E_0;s_{n-1})=0$.
\end{proof}

\begin{lemma}
  \label{o=0}
  The following statements are equivalent:

  \begin{enumerate}
    \item $o_n(E,E_0;s_{n-1})=0$.

    \item The section $s_{n-1}$ is uniformly fiberwise homotopic to a bounded uniformly continuous section $s_n\colon B\to E$ satisfying $s_n(B_n)\subset\mathrm{Int}_{\epsilon_n}(B_n)$ for some $\epsilon_n>0$, by a uniform fiberwise homotopy $H\colon B\times[0,1]\to E$ which is stationary on $B_{n-2}$ and satisfies $H(B_{n-1}\times[0,1])\subset\mathrm{Int}_{\epsilon_n}(E_0)$.
  \end{enumerate}
\end{lemma}

\begin{proof}
  This follows immediately from Lemmas \ref{continuity} and \ref{Ascoli-Arzela}.
\end{proof}

For $i=0,1$, let $s^i_{n-1}\colon B\to E$ be bounded uniformly continuous sections of $p\colon E\to B$ satisfying $s_{n-1}^i\vert_{B_{n-2}}=s_{n-1}\vert_{B_{n-2}}$ and $s_{n-1}^i(B_{n-1})\subset\mathrm{Int}_{\epsilon_{n-1}}(E_0)$. Suppose there is a uniform fiberwise homotopy $H\colon B\times[0,1]\to E$ from $s_{n-1}^0$ to $s_{n-1}^1$ such that $H(B_{n-2}\times[0,1])\subset\mathrm{Int}_\epsilon(E_0)$. Consider the relative fiber bundle
\[
  (F,F_0)\to(E\times[0,1],E_0\times[0,1])\xrightarrow{p\times 1}B\times[0,1].
\]
By Lemma \ref{pullback bundle}, this bundle is uniform. Consider the map
\[
  s_H\colon B\times[0,1]\to E\times[0,1],\quad(x,t)\mapsto(H(x,t),t)
\]
which defines a bounded uniformly continuous section of $p\times 1\colon E\times[0,1]\to B\times[0,1]$. Hence we obtain a cocycle $o_n(E\times[0,1],E_0\times[0,1];s_H)\in\mathrm{Cell}^n_\mathrm{ub}(B\times[0,1];\underline{\pi_n(F,F_0)})$. Define
\[
  d_{n-1}(E,E_0;H)\colon\Lambda_{n-1}\to\pi_n(F,F_0),\quad\lambda\mapsto o_n(E\times[0,1],E_0\times[0,1];s_H)(e_\lambda\times(0,1)),
\]
where $e_\lambda\in\Lambda_{n-1}$. Since $o_n(E\times[0,1],E_0\times[0,1];s_H)$ is bounded, the same holds for $d_{n-1}(E,E_0;H)$, so that is $d_{n-1}(E,E_0;H)\in\mathrm{Cell}^{n-1}_\mathrm{ub}(B;\underline{\pi_n(F,F_0)})$.

\begin{lemma}
  \label{d1}
  $\delta d(E,E_0;H)=o_n(E,E_0;s_{n-1}^1)-o_n(E,E_0;s_{n-1}^0)$.
\end{lemma}

\begin{proof}
  Write $o_n=o_n(E\times[0,1],E_0\times[0,1];s_H)$. By Lemma \ref{cocycle}, $o_n$ is a cocycle. Hence, for any $(n-1)$-cell $e$ of $B$,
  \begin{align*}
    0&=(\delta o_n)((0,1)\times e)\\
    &=o_n(\partial((0,1)\times e))\\
    &=o_n(1\times e)-o_n(0\times e)-o_n((0,1)\times\partial e)\\
    &=(o_n(E,E_0;s_{n-1}^1)-o_n(E,E_0;s_{n-1}^0)-\delta d_{n-1}(E,E_0;F))(e).
  \end{align*}
  Therefore, $\delta\, d(E,E_0;H)=o_n(E,E_0;s_{n-1}^1) - o_n(E,E_0;s_{n-1}^0)$.
\end{proof}

\begin{lemma}
  \label{d2}
  For any $d\in\mathrm{Cell}_\mathrm{ub}^{n-1}(B;\underline{\pi_n(F,F_0)})$, there is a uniform fiberwise homotopy $H\colon B\times[0,1]\to E$ such that
  \[
    H(-,0)=s_{n-1}\quad\text{and}\quad H((B_{n-2}\times[0,1])\cup(B_{n-1}\times 1))\subset\mathrm{Int}_\epsilon(E_0)
  \]
  for some $\epsilon>0$ and
  \[
    d_{n-1}(E,E_0;H)=d.
  \]
\end{lemma}

\begin{proof}
  Since $d\colon \Lambda_{n-1}\to\underline{\pi_n(F,F_0)}$ is bounded, its image is a finite subset $\{a_1,\ldots,a_m\}\subset\pi_n(F,F_0)$. Because $E_0$ is open in $E$ and $S^{n-1}$ is compact, there exists $\epsilon>0$ such that for each $i=1,\ldots,m$, there is a representative $\bar{a}_i\colon(D^n,S^{n-1})\to(F,F_0)$ of $a_i\in\pi_n(F,F_0)\cong\pi_n(E,E_0)$ satisfying $\bar{a}_i(S^{n-1})\subset\mathrm{Int}_\epsilon(E_0)$.

  Let $\{\mu_1,\ldots,\mu_k\}\subset\Lambda_{n-1}$ be as in Lemma \ref{Ascoli-Arzela}, where $\Lambda_n$ is replaced by $\Lambda_{n-1}$. For $i=1,\ldots,m$ and $\lambda\in\Lambda_{n-1}$, define
  \[
    H_{\lambda,i}\colon D^{n-1}\times[0,1]\to(D^{n-1}\times[0,1])\vee D^n\xrightarrow{\mathrm{proj}\vee 1}D^{n-1}\vee D^n\xrightarrow{\phi_\lambda\vee 1}B\vee D^n\xrightarrow{s+\bar{a}_i}E,
  \]
  where the first map is the pinch map. For each $\lambda\in\Lambda_{n-1}$, choose $i_\lambda\in\{1,\ldots,m\}$ with $d(e_\lambda)=d_{i_\lambda}$. Then the maps $H_{\lambda,i_\lambda}$ patch together to form a continuous map $\widehat{H}\colon(B\times 0)\cup(B_{n-1}\times[0,1])\to E$ satisfying
  \[
    \widehat{H}(-,0)=s_{n-1}\quad\text{and}\quad\widehat{H}((B_{n-2}\times[0,1])\cup(B_{n-1}\times 1))\subset\mathrm{Int}_\epsilon(E_0).
  \]
  Since the family $\{\phi_\lambda\}_{\lambda\in\Lambda_{n-1}}$ is equicontinuous, so is $\{H_{\lambda,i_\lambda}\}_{\lambda\in\Lambda_{n-1}}$. By Lemma \ref{continuity}, $\widehat{H}$ is uniformly continuous and fiberwise, and it extends to a fiberwise uniform homotopy $H\colon B\times[0,1]\to E$. By construction,
  \[
    H(-,0)=s_{n-1},\quad H((B_{n-2}\times[0,1])\cup(B_{n-1}\times 1))\subset\mathrm{Int}_\epsilon(E_0)
  \]
  and for all $\lambda\in\Lambda_{n-1}$,
  \[
    d_{n-1}(E,E_0;H)(e_\lambda)=a_{i_\lambda}.
  \]
  Hence, $d_{n-1}(E,E_0;H)=d$, completing the proof.
\end{proof}

Finally, we are ready to prove Theorem \ref{obstruction}.

\begin{proof}
  By Lemma \ref{cocycle}, we may define the obstruction class
  \[
    \mathfrak{o}_n(E,E_0;s_{n-1})=[o_n(E,E_0;s_{n-1})]\in H^n_\mathrm{ub}(B;\underline{\pi_n(F,F_0)}).
  \]
  Lemmas \ref{o=0}, \ref{d1}, and \ref{d2} together show that $\mathfrak{o}_n(E,E_0;s)$ has the desired property. The naturality of $\mathfrak{o}_n(E,E_0;s_{n-1})$ follows directly from its construction.
\end{proof}


\section{Uniform Lefschetz fixed-point theorem}\label{Uniform Lefschetz fixed-point theorem}

In this section, we prove Theorem \ref{main} by applying the obstruction theory in the previous section. Furthermore, we generalize Theorem \ref{main} to the coincidence theorem.


\subsection{Uniform manifold}

We begin by considering the notion of uniformity of manifolds, which will be required in Theorem \ref{main}.

\begin{definition}
  \label{uniform manifold}
  A connected Riemannian manifold $M$ is said to be \emph{uniform} if there is an equicontinuous family of homeomorphisms $\{h_{x,y}\colon M\to M\}_{x,y\in M}$ satisfying $f_{x,y}(x)=y$.
\end{definition}

We now present examples of uniform manifolds. Let $M$ be a connected complete Riemannian manifold. If there are a subset $\Gamma\subset M$ with $N_r(\Gamma)=M$ for some $r>0$ and an equicontinuous family of uniformly continuous homeomorphisms $\{h_{x,y}\colon M\to M\}_{x,y\in\Gamma}$ satisfying $f_{x,y}(x)=y$, then $M$ is called \emph{coarsely uniform}. Clearly, any uniform manifold is coarsely uniform, although the converse does not hold in general.

\begin{lemma}
  \label{coarsely uniform}
  If a connected complete Riemannian manifold of bounded geometry is coarsely uniform, then it is uniform.
\end{lemma}

\begin{proof}
  Let $M$ be such a manifold. Then there are a subset $\Gamma\subset M$ such that $N_r(\Gamma)=M$ for some $r>0$ and an equicontinuous family of homeomorphisms $\{h_{x,y}\colon M\to M\}_{x,y\in\Gamma}$ with $h_{x,y}(x)=y$. Since $M$ has bounded geometry, its injectivity radius is positive, say $\epsilon>0$, and the diffeomorphisms given by the exponential map on $N_\epsilon(x)$ for $x\in M$ form an equicontinuous family. Consequently, we obtain an equicontinuous family of homeomorphisms $\{g_{x,y}\colon M\to M\}_{d(x,y)<\epsilon/2}$ such that $g_{x,y}(x)=y$ and each $g_{(x,y)}$ is stationary outside $N_\epsilon(x)$. By composing these maps, we get an equicontinuous family of homeomorphisms $\{\bar{g}_{x,y}\colon M\to M\}_{d(x,y)<r}$. Now define
  \[
    \hat{h}_{x,y}=\bar{g}_{x,x_0}\circ h_{x_0,y_0}\circ\bar{g}_{y_0,y}
  \]
  for $x,y\in M$, where we choose any $x_0,y_0\in\Gamma$ are chosen to satisfy $d(x,x_0)<r$ and $d(y,y_0)<r$. The family $\{\hat{h}_{x,y}\colon M\to M\}_{x,y\in M}$ is equicontinuous, completing the proof.
\end{proof}

\begin{proposition}
  \label{action}
  If a connected complete Riemannian manifold of bounded geometry admits a cocompact isometric action of a discrete group, then it is uniform.
\end{proposition}

\begin{proof}
  Let $M$ be such a manifold, and let $G$ act on $M$ by isometries cocompactly. Denote by $p\colon M\to M/G$ the projection. Choose a point $x_0\in M/G$ and set $\Gamma=p^{-1}(x_0)$. Then as $M/G$ is compact and the action of $G$ is isometric, there exists $r>0$ such that $N_r(\Gamma)=M$. The family $\{g\colon M\to M\}_{g\in G}$ is equicontinuous, hence $M$ is coarsely uniform. By Lemma \ref{coarsely uniform}, the manifold $M$ is therefore uniform.
\end{proof}

\begin{remark}
  If a manifold is oriented, then in Definition \ref{uniform manifold} and in the notion of course uniformity above, we restrict homeomorphisms $h_{x,y}$ to be orientation preserving. It is straightforward to verify that Lemma \ref{coarsely uniform} and Proposition \ref{action} remain valid in the oriented case, provided the action of $G$ in Proposition \ref{action} is orientation preserving.
\end{remark}

Let $M$ be a connected manifold, and let $p\colon M\times M\to M$ denote the first projection. Then there is a relative fiber bundle
\begin{equation}
  \label{Fadell-Neuwirth}
  (M,M-x_0)\to(M\times M,M\times M-\Delta)\xrightarrow{p}M
\end{equation}
due to Fadell and Neuwirth \cite{FN}, where $x_0$ is a basepoint of $M$ and $\Delta$ denotes the diagonal subset of $M\times M$. However, this is not necessarily uniform, since its trivializations need not be equicontinuous.

\begin{lemma}
  \label{FN uniform}
  If $M$ is a uniform connected complete Riemannian manifold of bounded geometry, then the relative fiber bundle \eqref{Fadell-Neuwirth} is good and uniform.
\end{lemma}

\begin{proof}
  The first projection $p\colon M\times M\to M$ is uniformly continuous, so condition (1) of Definition \ref{uniform fiber bundle} is satisfied. Since $M$ has bounded geometry, the injectivity radius is positive, say $r>0$. Then $\{N_r(x)\}_{x\in M}$ is an open cover of $M$ with Lebesgue number $r$, verifying condition (2) of Definition \ref{uniform fiber bundle}. Take a trivialization
  \[
    h\colon(p^{-1}(N_\delta(x_0)),p^{-1}(N_\delta(x_0))\cap(M\times M-\Delta))\xrightarrow{\cong}N_\delta(x_0)\times(M,M-x_0)
  \]
  for some $\delta>0$. Since $M$ is uniform, there is an equicontinuous family of uniformly continuous homeomorphisms $\{f_{x,y}\colon M\to M\}_{x,y\in M}$ with $f_{x,y}(x)=y$. Then there exists $\epsilon>0$ such that, for any $x\in M$,  $f_{x,x_0}(N_\epsilon(x))\subset N_\delta(x_0)$. For $x\in M$, define $h_x$ as the composite
  \begin{multline*}
    (p^{-1}(N_\epsilon(x)),p^{-1}(N_\epsilon(x))\cap(M\times M-\Delta))\\
    \xrightarrow{h\circ(h_{x,x_0}\times h_{x,x_0})}h_{x,x_0}(N_\epsilon(x))\times(M,M-x_0)\xrightarrow{h_{x,x_0}^{-1}\times 1}N_\epsilon(x)\times(M,M-x_0).
  \end{multline*}
  Then $\{h_x\}_{x\in M}$ and $\{h_x^{-1}\}_{x\in M}$ are uniformly equicontinuous families of homeomorphisms. Hence condition (3) of Definition \ref{uniform fiber bundle} holds, and the relative fiber bundle \eqref{Fadell-Neuwirth} is uniform.

  As $M$ is Hausdorff, $M\times M-\Delta$ is open in $M\times M$, satisfying condition (1) of Definition \ref{good}. Since $M$ is a complete Riemannian manifold, it is a complete and proper metric space. Moreover, if $\dim M\ge 2$, $M-x_0$ is nonempty and path-connected, satisfying condition (2) of Definition \ref{good}. As in Example \ref{triangulation}, $M$ has a biLipschitz homeomorphic to a uniform simplicial complex. Given $\epsilon>0$, by applying finitely many barycentric subdivisions, we may assume that each simplex of $M$ has diameter less than $\epsilon$. Because the cover $\{U_i\}_{i\in I}$ of $M$ has a positive Lebesgue number, condition (3) of Definition \ref{good} is also satisfied. Hence the uniform relative fiber bundle \eqref{Fadell-Neuwirth} is good.
\end{proof}

\begin{lemma}
  \label{admissible assumption}
  Let $M$ be a simply-connected complete Riemannian $n$-manifold for $n\ge 2$. Then $(M,M-x_0)$ is $k$-admissible for $k=1,\ldots,n$ and $\pi_m(M,M-x_0)=0$ for $m=0,\ldots,n-1$.
\end{lemma}

\begin{proof}
  Since $n\ge 2$, $M-x_0$ is path-connected, satisfying condition (1) of Definition \ref{admissible}. Clearly, $\pi_k(M,M-x_0)$ is finitely generated for $k=1,\ldots,n$. As $M$ is simply-connected, $\pi_1(M-x_0)=0$ for $n\ge 3$, and hence the action of $\pi_1(M-x_0)$ on $\pi_k(M,M-x_0)$ is trivial for $k=1,\ldots,n$. When $n=2$, since $M$ is simply-connected, it is diffeomorphic to $S^2$ or $\R^2$, and again the action of $\pi_1(M-x_0)$ on $\pi_k(M,M-x_0)$ is trivial for $k=1,2$. Thus condition (2) of Definition \ref{admissible} holds. As observed in Remark \ref{admissible condition}, condition (3) also holds. Therefore, $(M,M-x_0)$ is $k$-admissible for $k=1,\ldots,n$. Finally, since $\dim M\ge 2$, we have $\pi_m(M,M-x_0)=0$ for $m=0,\ldots,n-1$, completing the proof.
\end{proof}


\subsection{Uniform Lefschetz class}

For the rest of this section, let $M$ be a uniform simply-connected complete Riemannian $n$-manifold of bounded geometry with $n\ge 2$. We record an easy but important lemma, and omit the proof as it is straightforward.

\begin{lemma}
  \label{section-map}
  A bounded uniformly continuous section $s\colon M\to M\times M$ of $p\colon M\times M$ defines a uniformly continuous map
  \[
    f_s\colon M\xrightarrow{s}M\times M\xrightarrow{q}M
  \]
  satisfying $d(f,1)<\infty$, where $q$ is the second projection. Moreover, if two bounded uniformly continuous sections $s_0,s_1\colon M\to M\times M$ of $p\colon M\times M\to M$ are uniformly fiberwise homotopic if and only if $f_{s_0}$ and $f_{s_1}$ are uniformly homotopic.
\end{lemma}

We now define the uniform Lefschetz class.

\begin{definition}
  The \emph{uniform Lefschetz class} of a uniformly continuous map $f\colon M\to M$ with $d(f,1)<\infty$ is defined by
  \[
    \mathscr{L}(f)=\mathfrak{o}_n(M\times M,M\times M-\Delta;s_f)\frown[M]\in H_0^\mathrm{ub}(M).
  \]
\end{definition}

By Lemmas \ref{FN uniform}, the relative fiber bundle \eqref{Fadell-Neuwirth} is good and uniform. By Lemma \ref{admissible assumption}, $(M,M-x_0)$ is $n$-admissible.
Since the homeomorphisms $h_{x,y}$ are orientation preserving, condition (2) of Theorem \ref{obstruction} is satisfied.
Hence the obstruction class $\mathfrak{o}_n(M\times M,M\times M-\Delta;s_f)$ can be defined, and thus the uniform Lefschetz class $\mathscr{L}(f)$.
By Proposition \ref{0-dim}, we may regard $\mathscr{L}(f)$ as an element of $H_0^\mathrm{uf}(M)$.

\begin{theorem}
  \label{main late}
  Let $M$ be a uniform simply-connected complete Riemannian manifold of bounded geometry with $\dim M\ge 2$, and let $f\colon M\to M$ be a uniformly continuous map with $d(f,1)<\infty$. Then $\mathscr{L}(f)=0$ if and only if $f$ is uniformly homotopic to a strongly fixed-point-free uniformly continuous map.
\end{theorem}

\begin{proof}
  Clearly, $s_f$ is uniformly fiberwise homotopic to a bounded uniformly continuous section $s_0\colon M\to M\times M$ satisfying $s_0(M_0)\subset\mathrm{Int}_{\epsilon_0}(M\times M-\Delta)$ for some $\epsilon_0>0$. Suppose that for $k\le n-2$, $s_f$ is uniformly fiberwise homotopic to a bounded uniformly continuous section $s_k$ satisfying $s_k(M_k)\subset\mathrm{Int}_{\epsilon_k}(M\times M-\Delta)$ for some $\epsilon_k>0$. By Lemma \ref{admissible assumption}, $(M,M-x_0)$ is $(k+1)$-admissible, so the obstruction class $\mathfrak{o}_{k+1}(M\times M,M\times M-\Delta;s_k)$ is defined. Again by Lemma \ref{admissible assumption}, $\pi_{k+1}(M,M-x_0)=0$ for $k\le n-2$. Hence, by Theorem \ref{obstruction}, the section $s_k$ is uniformly fiberwise homotopic to a bounded uniformly continuous section $s_{k+1}$ satisfying $s_{k+1}(M_{k+1})\subset\mathrm{Int}_{\epsilon_{k+1}}(M\times M-\Delta)$ for some $\epsilon_{k+1}>0$. Consequently, $s_f$ is uniformly fiberwise homotopic to a bounded uniformly continuous section $s_{n-1}$ satisfying $s_{n-1}(M_{n-1})\subset\mathrm{Int}_{\epsilon_{n-1}}(M\times M-\Delta)$ for some $\epsilon_{n-1}>0$. By construction,
  \[
    \mathfrak{o}_n(M\times M,M\times M-\Delta;s_{n-1})=\mathfrak{o}_n(M\times M,M\times M-\Delta;s_f).
  \]
  The statement now follows from Theorems \ref{Poincare duality} and \ref{obstruction}.
\end{proof}

\begin{remark}
  From the discussion above, we see that the only use of the smoothness of a manifold is to guarantee the existence of a biLipschitz triangulation by a uniform simplicial complex in Example \ref{triangulation}. Therefore, Theorem \ref{main} remains valid for a topological manifold which is a metric space and admits a biLipschitz triangulation by a uniform simplicial complex.
\end{remark}

\begin{corollary}
  \label{amenable Lefschetz}
  Let $M$ be a uniform simply-connected complete Riemannian $n$-manifold of bounded geometry with $n\ge 2$. If $M$ is nonamenable, then any uniformly continuous map $f\colon M\to M$ with $d(f,1)<\infty$ is uniformly homotopic to a strongly fixed-point-free uniformly continuous map.
\end{corollary}

\begin{proof}
  If $M$ is nonamenable, then by Proposition \ref{amenable H_0}, $H_0^\mathrm{uf}(M)=0$, hence $\mathscr{L}(f)=0$. The claim follows from Theorem \ref{main}.
\end{proof}

\begin{corollary}
  \label{corollary}
  Let $M$ be the universal cover of a closed connected $n$-manifold whose fundamental group is nonamenable. Then any uniformly continuous map $f\colon M\to M$ with $d(f,1)<\infty$ is uniformly homotopic to a strongly fixed-point-free uniformly continuous map.
\end{corollary}

\begin{proof}
  By Example \ref{Galois cover space}, $M$ is nonamenable. Thus for $\dim M\ge 2$, the claim follows from Corollary \ref{amenable Lefschetz}. The $\dim M=1$ case follows from the fact that only $\Z$ acts freely and cocompactly on $\R$.
\end{proof}


\subsection{Lefschetz coincidence theorem}\label{Lefschetz coincidence theorem}

The classical Lefschetz coincidence theorem states that for maps $f,g\colon X\to Y$ between finite complexes, the Lefschetz number $L(f,g)\ne 0$, then there exists $x\in X$ satisfying $f(x)=g(x)$. Fadell \cite{F} showed that the converse also holds whenever $Y$ is a simply-connected closed manifold. Note that the Lefschetz fixed-point theorem is the special case of the Lefschetz coincidence theorem such that $L(f)=L(f,1)$. We generalize Theorem \ref{main} to the coincidence theorem.

Let $M$ be a connected manifold. For a map $f\colon M\to M$, define
\[
  \Delta_f=\{(x,f(x))\in M\times M\}.
\]

\begin{lemma}
  \label{FN f uniform}
  Let $M$ be a uniform connected complete Riemannian manifold of bounded geometry, and let $f\colon M\to M$ be a uniformly continuous map. Then the relative fiber bundle
  \begin{equation}
    \label{Fadell-Neuwirth f}
    (M,M-x_0)\to(M\times M,M\times M-\Delta_f)\xrightarrow{p}M,
  \end{equation}
  where $p\colon M\times M\to M$ denotes the first projection, is uniform and good.
\end{lemma}

\begin{proof}
  The relative fibration \eqref{Fadell-Neuwirth f} is the pullback of \eqref{Fadell-Neuwirth} along the map $f$.
  Hence, by Lemmas \ref{pullback bundle} and \ref{FN uniform}, it is a uniform relative fiber bundle.
  Moreover, the argument in the final part of the proof of Lemma \ref{FN uniform} applies verbatim to show that \eqref{Fadell-Neuwirth f} also satisfies the conditions of a good relative fiber bundle.
\end{proof}

For maps $f,g\colon X\to X$ of a metric space $X$, define
\[
  d(f,g)=\sup_{x\in X}d(f(x),g(x)).
\]

\begin{definition}
  Let $M$ be a uniform simply-connected complete Riemannian $n$-manifold of bounded geometry with $n\ge 2$. The \emph{uniform Lefschetz class} of uniformly continuous maps $f,g\colon M\to M$ satisfying $d(f,g)<\infty$ is defined by
  \[
    \mathscr{L}(f,g)=\mathfrak{o}_n(M\times M,M\times M-\Delta_f;s_g)\frown[M]\in H_0^\mathrm{ub}(M),
  \]
  where $s_g\colon M\to M\times M$ is the section $s_g(x)=(x,g(x))$ of $p\colon M\times M\to M$.
\end{definition}

As in the case of the uniform Lefschetz class $\mathscr{L}(f)$, one can verify that the uniform Lefschetz class $\mathscr{L}(f,g)$ is well defined. We say that maps $f,g\colon M\to M$ are \emph{strongly coincidence-free} if there exists $\epsilon>0$ such that $d(f(x),g(x))\ge\epsilon$ for all $x\in X$. Analogously to Theorem \ref{main}, we obtain the following uniform version of Fadell's result \cite{F}.

\begin{theorem}
  Let $M$ be a uniform simply-connected complete Riemannian $n$-manifold of bounded geometry with $n\ge 2$, and let $f,g\colon M\to M$ be uniformly continuous maps with $d(f,g)<\infty$. Then $\mathscr{L}(f,g)<\infty$ if and only if $f$ and $g$ are uniformly homotopic to strongly coincidence-free maps.
\end{theorem}

Analogously to Corollary \ref{corollary}, we obtain the following.

\begin{corollary}
  Let $M$ be the universal cover of a closed connected manifold with nonamenable $\pi_1$. Then any uniformly continuous maps $f,g\colon M\to M$ with $d(f,g)<\infty$ are uniformly homotopic to strongly coincidence-free maps.
\end{corollary}


\section{Localization}\label{Localization}

In this section, we prove Theorem \ref{Lefschetz-Hopf} by adopting the approach of \cite{KKT1}, which deals with counting infinitely many points on a Galois covering of a closed manifold. We then return to the Poincar\'e-Hopf theorem for such coverings, studied in \cite{KKT1}, and improve it in the uniformly continuous setting.

\subsection{Lefschetz-Hopf theorem}

Let $G\to M\to N$ be a Galois covering, where $N$ is a closed connected $n$-manifold. We lift a metric of $N$ to $M$, so that $M$ is equipped with a $G$-invariant metric and has bounded geometry. Fix a triangulation of $N$, and lift it to a $G$-invariant triangulation of $M$. Following \cite{KKT1,KKT2}, we recall the notion of a fundamental domain. For each open $n$-simplex of $N$, choose one of its lift of to $M$; for each $k$-simplex of $N$ with $k\le n-1$, choose one of its lift to $M$ in such a way that it is contained in the closure of some chosen open $n$-simplex of $M$. Define the \emph{fundamental domain} $K$ to be the union of these chosen open simplices of $M$. We fix such a fundamental domain $K$. By definition, we have the following.

\begin{lemma}
  \label{fundamental domain}
  A fundamental domain $K$ is relatively compact and satisfies
  \[
    M=\coprod_{g\in G}gK.
  \]
\end{lemma}

In \cite{KKT1}, the notions of tameness and strongly tameness of diffeomorphisms were introduced in the study of the Poincar\'e-Hopf theorem for a Galois covering of a closed manifold. These definitions extend naturally to self-maps as follows (cf. \cite[Definitions 5.5 and 5.6]{KKT1}).

\begin{definition}
  \label{tame}
  Let $G\to M\to N$ be a Galois covering, where $N$ is a closed connected manifold. A map $f\colon M\to M$ is called \emph{tame} if there exist $\delta>0$ and $\epsilon>0$ such that:

  \begin{enumerate}
    \item For any distinct fixt-points $x,y\in\mathrm{Fix}(f)$, one has $N_\delta(x)\cap N_\delta(y)=\emptyset$.

    \item For any $x\in M-N_\delta(\mathrm{Fix}(f))$, one has $d(x,f(x))\ge\epsilon$.
  \end{enumerate}

  \noindent If, in addition, the following condition holds, then $f$ is called \emph{strongly tame}.

  \begin{enumerate}
    \item[(3)] For every $x\in\mathrm{Fix}(f)$, there is an $n$-simplex $\sigma$ of $M$ such that $N_\delta(x)\subset\sigma$.
  \end{enumerate}
\end{definition}

We can now prove the Lefschetz-Hopf theorem.

\begin{theorem}
  \label{Lefschetz-Hopf late}
  Let $M$ be the universal cover of a closed connected manifold $N$, and let $\pi=\pi_1(N)$. For a strongly tame uniformly continuous map $f\colon M\to M$ with $d(f,1)<\infty$, $\mathscr{L}(f)$ is represented by the map
  \[
    \pi\to\Z,\quad g\mapsto\sum_{x\in gK\cap\Fix(f)}\ind_x(f).
  \]
\end{theorem}

\begin{proof}
  Let $x\in\mathrm{Fix}(f)$. By condition (3) of Definition \ref{tame}, there is an $n$-simplex $\sigma_x$ of $M$ satisfying $N_\delta(x)\subset\sigma_x$. Observe that
  \[
    o_n(M\times M,M\times M-\Delta;s_f)\frown\sigma_x=av
  \]
  for some $a\in\Z$ and a vertex $v$ of $\sigma_x$. By connecting $v$ to the barycenter $v_x$ of $v_x$ of $\sigma_x$ by a linear $1$-simplex, we see that $av$ is homologous to $av_x$. Since we are working with a uniform simplicial complex, by applying a finite number of barycentric subdivisions, we may assume that every $n$-simplex contains at most one fixed-point. By \cite[5.4 Proposition]{D}, we also have $a=\ind_x(f)$. Hence,
  \[
    \left[\sum_{x\in\mathrm{Fix}(f)\cap gK}o_n(M\times M,M\times M-\Delta;s_f)\frown\sigma_x\right]=\left[\sum_{x\in\mathrm{Fix}(f)\cap gK}\ind_x(f)v_x\right].
  \]
  As $x$ ranges over all fixed-points of $f$, the collection of the above linear $1$-simplices forms a uniform family of singular $1$-simplices in $M$. Consequently,
  \[
    \left[\sum_{x\in\mathrm{Fix}(f)}o_n(M\times M,M\times M-\Delta;s_f)\frown\sigma_x\right]=\left[\sum_{x\in\mathrm{Fix}(f)}\ind_x(f)v_x\right].
  \]
  On the other hand, by condition (2) of Definition \ref{tame},
  \[
    \mathscr{L}(f)=\left[\sum_{x\in\mathrm{Fix}(f)}o_n(M\times M,M\times M-\Delta;s_f)\frown\sigma_x\right].
  \]
  Therefore, by Proposition \ref{0-dim} and Lemmas \ref{H_0 covering} and \ref{fundamental domain}, the claim follows.
\end{proof}

Let $f\colon M\to M$ be as in Theorem \ref{main}. Then, whenever $\mathscr{L}(f)\ne 0$, $f$ is not uniformly homotopic to a strongly fixed-point-free uniformly continuous map. In other words, any map that is uniformly homotopic to $f$ has a fixed-point on $M$ or at infinity. This result can be strengthened for Galois coverings on closed manifolds as follows.

\begin{corollary}
  \label{infinitely many fixed-points}
  Let $M$ be as in Theorem \ref{Lefschetz-Hopf late}, and let $f\colon M\to M$ be a uniformly continuous map with $d(f,1)<\infty$. If $M$ is noncompact and $\mathscr{L}(f)\ne 0$, then any strongly tame map that is uniformly homotopic to $f$ has infinitely many fixed-points on $M$.
\end{corollary}

\begin{proof}
  Suppose that a strongly tame map $h\colon M\to M$ is uniformly homotopic to $f$. Then $h$ is uniformly continuous and satisfies $d(h,1)<\infty$. Hence, the uniform Lefschetz class $\mathscr{L}(h)$ is defined , and by construction,
  \[
    \mathscr{L}(h)=\mathscr{L}(f).
  \]
  Consequently, $\mathscr{L}(h)\ne 0$. Since $M$ is noncompact, $\pi$ is infinite. Hence, by Theorem \ref{Lefschetz-Hopf late}, it follows that
  \[
    \sum_{x\in gK\cap\Fix(h)}\ind_x(h)\ne 0
  \]
  for infinitely many $g\in\pi$. Therefore, by Proposition \ref{coinvariants finite}, the claim follows.
\end{proof}

\begin{corollary}
  \label{finitely many fixed-points}
  Let $M$ be as in Theorem \ref{Lefschetz-Hopf late}, and let $f\colon M\to M$ be a uniformly continuous map satisfying $d(f,1)<\infty$. If $M$ is noncompact and $f$ has only finitely many fixed-points, then $f$ is uniformly homotopic to a strongly fixed-point-free map.
\end{corollary}

\begin{proof}
  Suppose that $M$ is the universal cover of $N$. If $\pi_1(N)$ is nonamenable, then the claim follows from Corollary \ref{corollary}. Suppose that $\pi_1(N)$ is amenable. Since $f$ has only finitely many fixed-points, it is strongly tame. Hence the claim follows from Corollary \ref{infinitely many fixed-points}.
\end{proof}

\begin{example}
  \label{connected sum}
  Let $S^2=\{(x,y,z)\in\R^3\mid x^2+y^2+z^2=1\}$, and set $p=(0,0,1)$ and $q=(0,0,-1)$. Define
  \[
    M=\cdots\#(S^2\times S^2)\#(S^2\times S^2)\#\cdots,
  \]
  the infinite connected sum of copies of $S^2\times S^2$ taken around the points $(p,p)$ and $(q,q)$. The group $\Z$ acts on $M$ by shifting the summands so that the projection $M\to M/\Z$ is a Galois covering, where $M/\Z$ is a closed connected manifold. We may choose a fundamental domain $K\subset M$ to be a connected summand $S^2\times S^2$ with a small closed disc around $(p,p)$ and a small open disc around $(q,q)$ removed. The rotation by $\pi$ around the $z$-axis $S^2\to S^2$ defines a $\Z$-invariant map $f\colon M\to M$ such that, for every $g\in\Z$,
  \[
    gK\cap\Fix(f)=\{(p,q),(q,p)\}\text{ with }\ind_x((p,q))=\ind_y((q,p))=1.
  \]
  By construction, the map $f$ satisfies $d(f,1)<\infty$. Thus, the uniform Lefschetz class $\mathscr{L}(f)$ is defined. By Theorem \ref{Lefschetz-Hopf late}, it is represented by the constant function $2\cdot\mathbbm{1}\colon\Z\to\Z$, where $\mathbbm{1}$ denotes the constant function on $\Z$ with value $1$. Since $\Z$ is amenable, there exists a $\Z$-invariant positive linear map $\mu\colon\ell^\infty(\Z)\to\R$ satisfying $\mu(\mathbbm{1})=1$. This induces a linear map
  \[
    \bar{\mu}\colon\ell^\infty(\Z)_\Z\to\R
  \]
  such that $\bar{\mu}([\mathbbm{1}])=1$. Hence, $\bar{\mu}(\mathscr{L}(f))=2\ne 0$, yielding $\mathscr{L}(f)\ne 0$. Therefore, by Corollary \ref{infinitely many fixed-points}, any strongly tame map that is uniformly homotopic to $f$ has infinitely many fixed-points on $M$.
\end{example}


\subsection{Poincar\'e-Hopf theorem}\label{Poincare-Hopf theorem}

In \cite{KKT1}, the authors established the Poincar\'e-Hopf theorem for vector fields on Galois coverings of closed manifolds. We revisit this result in the setting of uniformly continuous vector fields.

Let $M$ be a smoothly uniform $n$-manifold of bounded geometry, in the sense that there are diffeomorphisms $h_{x,y}\colon M\to M$ for all $x,y\in M$ such that the families $\{dh_{x,y}\}_{x,y\in M}$ and $\{dh_{x,y}^{-1}\}_{x,y\in M}$ are equicontinuous. Hence, an argument analogous to that in Lemma \ref{FN uniform} shows that the relative fiber bundle
\[
  (T_{x_0}M,T_{x_0}M-0)\to(TM,TM-M)\xrightarrow{p}M
\]
is uniform and good, where $x_0$ is a basepoint of $M$. Clearly, $(T_{x_0}M,T_{x_0}M-0)$ is $n$-admissible, where $n=\dim M$. Since $M$ is oriented, condition (2) of Theorem \ref{obstruction} is satisfied, and thus the obstruction class $\mathfrak{o}_n(TM,TM-M;s_0)$ is defined, where $s_0\colon M\to TM$ denotes the zero section. Analogously to the classical compact case, we define the Euler class.

\begin{definition}
  Let $M$ be a smoothly uniform $n$-manifold of bounded geometry. The \emph{uniform bounded Euler class} of $M$ by
  \[
    e_\mathrm{ub}(M)=\mathfrak{o}_n(TM,TM-M;s_0)\in H^n_\mathrm{ub}(M;\Z).
  \]
\end{definition}

\begin{remark}
  One can define a notion of a uniform vector bundle in an appropriate sense, and the uniform bounded Euler class naturally extends to this setting.
\end{remark}

A vector field $v$ on a Riemannian manifold $M$ is said to be \emph{bounded} if
\[
  \sup_{x\in M}|v(x)|<\infty.
\]
Note that $v$ is bounded if and only if it is bounded as a section of the projection $p\colon TM\to M$. By Theorem \ref{obstruction}, we obtain the following (cf. \cite[Theorem 1]{W}).

\begin{proposition}
  Let $M$ be as above. Then there exists a nonvanishing bounded uniformly continuous vector field on $M$ if and only if $e_\mathrm{ub}(M)=0$.
\end{proposition}

We define the \emph{uniform bounded Euler characteristic} by
\[
  \chi_\mathrm{ub}(M)=e_\mathrm{ub}(M)\frown[M].
\]
We now compute the uniform bounded Euler characteristic of a Galois covering. Let $G\to M\to N$ be a Galois covering, where $N$ is a closed connected manifold. Fix a triangulation of $N$, and lift it to a $G$-invariant triangulation of $M$ as above. Choose a fundamental domain $K\subset M$.

\begin{proposition}
  \label{Euler characteristic}
  Let $G\to M\to N$ be a Galois covering, where $N$ is a closed connected manifold. Then
  \[
    \chi_\mathrm{ub}(M)=\chi(N)\mathbbm{1}\in\ell^\infty(G)_G.
  \]
\end{proposition}

\begin{proof}
  By definition, the restriction of $e_\mathrm{ub}(M)$ to $gK$ for each $g\in G$ coincides with the lift of the classical Euler class of $N$. Moreover, the fundamental class $[M]$ restricts to the lift of the fundamental class of $N$. Hence
  \[
    \mathfrak{o}_n(TM,TM-M;s_0)\vert_{gK}\frown[M]=e(N)\frown[N]=\chi(N)
  \]
  and the claim follows.
\end{proof}

The notions of tame and strongly tame vector fields on $M$ were introduced in \cite{KKT1}. We now adapt these definitions to the uniformly continuous setting (cf. \cite[Definitions 5.4 and 5.5]{KKT1}).

\begin{definition}
  A uniformly continuous vector field $v$ on $M$ is called \emph{tame} if there exist $\delta>0$ and $\epsilon>0$ such that:

  \begin{enumerate}
    \item For any distinct zeros $x,y\in\mathrm{Zero}(v)$, one has $N_\delta(x)\cap N_\delta(y)=\emptyset$.

    \item For any $x\in M-N_\delta(\mathrm{Zero}(v))$, one has $|v(x)|\ge\epsilon$.
  \end{enumerate}

  \noindent If, in addtion, the following condition holds, then $v$ is called \emph{strongly tame}.

  \begin{enumerate}
    \item[(3)] For every $x\in\mathrm{Zero}(v)$, there is an $n$-simplex $\sigma$ of $M$ such that $N_\delta(x)\subset\sigma$.
  \end{enumerate}
\end{definition}

Let $G\to M\to N$ be as above. The \emph{index} of a tame vector field $v$ on $M$ is defined by
\[
  \ind(v)\colon G\to\Z,\quad g\mapsto\sum_{x\in gK\cap\mathrm{Zero}(v)}\ind_x(v),
\]
where $\ind_x(v)$ denotes the local index of $v$ at $x\in\mathrm{Zero}(v)$. By the definition of the tameness of vector fields together with uniformness, $|gK\cap\mathrm{Zero}(v)|$ is bounded as $g$ ranges over $G$; hence, $\ind(v)$ is a bounded function (cf. \cite[Theorem 1.1]{KKT1}).

\begin{theorem}
  \label{Poincare-Hopf}
  Let $G\to M\to N$ be a Galois covering, where $N$ is a closed connected oriented manifold. For a bounded strongly tame uniformly continuous vector field $v$ on $M$, one has
  \[
    \ind(v)=\chi(N)\mathbbm{1}\quad\text{in }\ell^\infty(G)_G,
  \]
  where $\mathbbm{1}$ denotes the constant function on $G$ with value $1$.
\end{theorem}

\begin{proof}
  Since the vector field $v$ is bounded, it is uniformly fiberwise homotopic to $s_0$. Consequently,
  \[
    \mathfrak{o}_n(TM,TM-M;s_0)=\mathfrak{o}_n(TM,TM-M;v)
  \]
  The proof of Theorem \ref{Lefschetz-Hopf} applies verbatim, showing that $\chi_\mathrm{ub}(M)$ is represented by $\ind(v)\in\ell^\infty(G)$. Therefore, by Proposition \ref{Euler characteristic}, the claim follows.
\end{proof}

We present two corollaries of Theorem \ref{Poincare-Hopf}. The first is a uniformly continuous analogue of \cite[Theorem 1.2]{KKT1}.

\begin{corollary}
  Let $G\to M\to N$ be as in Theorem \ref{Poincare-Hopf}. If $G$ is an amenable infinite group and $\chi(N)\ne 0$, then any bounded uniformly continuous tame vector field on $M$ has infinitely many zeros.
\end{corollary}

\begin{proof}
  Suppose that a tame uniformly continuous vector field $v$ on $M$ has finitely many zeros. We can deform $v$ by a uniform homotopy to a strongly tame uniformly vector field whose zeros coincide those of $v$. Hence, we may assume that $v$ is strongly tame. By the argument analogous to Example \ref{connected sum} using an invariant mean, we obtain $\chi(N)\mathbbm{1}\ne 0$ in $\ell^\infty(G)_G$. Hence, by Theorem \ref{Poincare-Hopf}, $\ind(v)\ne 0$, yielding
  \[
    \sum_{x\in gK\cap\mathrm{Zero}(v)}\ind_x(v)\ne 0
  \]
  for infinitely many $g\in G$. Therefore by Proposition \ref{coinvariants finite}, the claim follows.
\end{proof}

The second corollary is new, as it does not follow from the smooth setting considered in \cite{KKT1}, where the method employed there cannot establish it since it does not involve the obstruction theoretic ideas introduced in the present paper.

\begin{corollary}
  Let $G\to M\to N$ be as in Theorem \ref{Poincare-Hopf}. If $\chi(N)\mathbbm{1}=0$ in $\ell^\infty(G)_G$, then there exists a nonvanishing bounded uniformly continuous vector field on $M$.
\end{corollary}

\begin{proof}
  If $\chi(N)\mathbbm{1}=0$ in $\ell^\infty(G)_G$, then, as in the proof of Theorem \ref{Poincare-Hopf},
  \[
    \mathfrak{o}_n(TM,TM-M;v)\frown[M]=0.
  \]
  Hence, by Theorem \ref{Poincare duality}, we have $\mathfrak{o}_n(TM,TM-M;v)=0$. By Theorem \ref{obstruction}, $v$ is uniformly homotopic to a nonvanishing vector field $w$. Clearly, $w$ is bounded and uniformly continuous, and thus the claim follows.
\end{proof}

\subsection{Obstruction for uniform null-homotopy}
Let $M$ be a complete Riemannian $n$-manifold of bounded geometry.
We consider the problem of determining when a given uniformly continuous map $f\colon M\to S^n$ is null-homotopic through a uniformly continuous homotopy.
This is regarded as an instance of the general obstruction-theoretic framework developed in Section \ref{Obstruction theory} as follows.
Let $p\colon M\times S^n\to M$ be the first projection, $E=M\times S^n$, $E_0=M\times U_0$ where $U_0=S^n\setminus\{\ast\}$, and $s(x)=(x,f(x))$.
Note that $(S^n,U_0)$ is $n$-admissible and the bundle $p\colon E\to M$ is trivial.
Then, as in the proof of Theorem \ref{main late}, our obstruction theory can be applied to this setting.
Since every uniformly continuous map into $U_0$ is uniformly homotopic to a constant map through a homotopy given by geodesics, we obtain the following theorem.
\begin{theorem}
  Let $M$ be a complete Riemannian $n$-manifold of bounded geometry with $n\ge2$ and $f\colon M\to S^n$ a uniformly continuous map.
  Then $f$ is null-homotopic through a uniformly continuous homotopy if and only if $\mathfrak{o}_n(M\times S^n,M\times U_0;s)=0$.
\end{theorem}

In the case of nonamenable manifolds, we have the following.

\begin{corollary}
  \label{Ferry-Weinberger nonamenable}
  Let $M$ be a nonamenable oriented complete Riemannian $n$-manifold of bounded geometry with $n\ge2$.
  Then any uniformly continuous map $f\colon M\to S^n$ is null-homotopic through a uniformly continuous homotopy.
\end{corollary}

\begin{proof}
  Since $M$ is nonamenable, we have $\mathfrak{o}_n(M\times S^n,M\times U_0;s)\in H^n_{\mathrm{ub}}(M;\mathbb{Z})=0$ by Proposition \ref{amenable H_0} and Theorem \ref{Poincare duality}.
  Therefore the claim follows.
\end{proof}

Ferry and Weinberger \cite[Theorem 2]{FW} considered the case where $M/G$ is a closed oriented manifold and $f\colon M\to S^n$ is the (uniformly continuous) map given by the composition of the covering map $M\to M/G$ with the degree-$1$ map $M/G\to S^n$ that collapses the $(n-1)$-skeleton.

\begin{corollary}
  \label{Ferry-Weinberger amenable}
  In the above setting, $f$ is null-homotopic through a uniformly continuous homotopy if and only if $G$ is nonamenable.
\end{corollary}
\begin{proof}
  If $G$ is nonamenable, then $f$ is null-homotopic through a uniformly continuous homotopy by the previous corollary.
  Let us consider the case when $G$ is amenable.
  Note that the inverse image $f^{-1}(\ast)$ intersects each fundamental domain of $M$ in a single point.
  Then, by an argument closely analogous to the proofs of Theorems \ref{Lefschetz-Hopf late} and \ref{Poincare-Hopf}, we obtain
  \[
  \mathfrak{o}_n(M\times S^n,M\times U_0;s)
  =
  \mathbbm{1}\ne0,
  \]
  since $G$ is amenable.
  Therefore, the claim follows.
\end{proof}

However, the original result by Ferry and Weinberger was for Lipschitz maps and homotopies.
This is verified as follows.

\begin{corollary}
  \label{Ferry-Weinberger amenable Lipschitz}
  In the previous setting, assume $f$ is a Lipschitz map.
  Then $f$ is null-homotopic through a Lipschitz homotopy if and only if $G$ is nonamenable.
\end{corollary}
\begin{proof}
  This immediately follows from Corollary \ref{Ferry-Weinberger amenable} and
  Corollary \ref{corollary equivalence uniform Lipschitz}, since a Lipschitz map is in
  particular uniformly continuous.
\end{proof}

\appendix

\section{Uniform simplicial approximation theorem}

In this appendix, we give the proof of the simplicial approximation theorem for uniformly continuous maps between simplicial complexes.
Most of the proof of the classical simplicial approximation theorem similarly works.
To fit it into our setting, we need the following lemmas.

\begin{lemma}
  \label{lemma_pullbackLebesguenumber}
  Let $\mathcal{U}$ be an open covering of a metric space with positive Lebesgue number and $f\colon X\to Y$ a uniformly continuous map between metric spaces.
  Then the inverse image $f^{-1}\mathcal{U}=\{f^{-1}(U)\}_{U\in\mathcal{U}}$ is also an open covering with positive Lebesgue number.
\end{lemma}
\begin{proof}
  Let $\epsilon>0$ be a Lebesgue number of $\mathcal{U}$.
  Since $f$ is uniformly continuous, there exists $\delta>0$ such that, for any $x,x'\in X$ with $d_X(x,x')<\delta$, $d_Y(f(x),f(x'))<\epsilon$.
  From this, if a subset $A\subset X$ has diameter less than $\delta$, the diameter of $f(A)$ is less than $\epsilon$ and hence $f(A)\subset U$ for some $U\in\mathcal{U}$, implying $A\subset f^{-1}(U)$.
  Thus $f^{-1}\mathcal{U}$ has a Lebesgue number $\delta$.
\end{proof}

\begin{lemma}
  \label{lemma_locallyLipschitz}
  Let $K$ be a uniform simplicial complex and $f\colon K\to Y$ a continuous map to a metric space.
  If the restriction of $f$ to each simplex has a Lipschitz constant less than $L$, then $f$ is a Lipschitz map with Lipschitz constant $L$.
\end{lemma}
\begin{proof}
  Since $K$ is uniform, there exists $N$ such that there are at most $N$ simplices in the $1$-neighborhood of each point.
  Then for any $x,x'\in K$ with $d_K(x,x')<1$, there exists a sequence of points $x=x_0,x_1,\ldots,x_m=x'$ such that these points are on a fixed geodesic connecting $x$ and $x'$ and ea ch segement between $x_{i-1}$ and $x_{i}$ is contained in a simplex of $K$.
  It is sufficient to show the inequality $d_Y(f(x),f(x'))\le Ld_K(x,x')$ for sufficiently close $x$ and $x'$.
  From this, we have
  \begin{align*}
    d_Y(f(x),f(x'))
    &\le
    d_Y(f(x_0),f(x_1))+\cdots+d_Y(f(x_{m-1}),f(x_m))\\
    &\le
    Ld_X(x_0,x_1)+\cdots+Ld_X(x_{m-1},x_m)
    =
    Ld_X(x,x').
  \end{align*}
  Thus $f$ has a Lipschitz constant $L$.
\end{proof}

\begin{lemma}
  \label{lemma_subdivisionLipschitz}
  Let $K$ be a uniform simplicial complex.
  Then the canonical homeomorphism $K\to\operatorname{sd}K$ is a Lipschitz map.
\end{lemma}
\begin{proof}
  Since $K$ is of finite dimension, the restriction to each simplex is one of the finitely many possibilities.
  Thus the map $K\to\operatorname{sd}K$ is Lipschitz by Lemma \ref{lemma_locallyLipschitz}.
\end{proof}

\begin{lemma}
  \label{lemma_simplicialmapLipschitz}
  Let $f\colon K\to K'$ be a uniformly continuous simplicial map between uniform simplicial complexes.
  Then $f$ is a Lipschitz map.
\end{lemma}
\begin{proof}
  Since $f$ is uniformly continuous, there exists $\delta>0$ such that $d_{K'}(f(x),f(x'))<1$ for any $x,x'\in K$ with $d_K(x,x')<\delta$.
  This implies that the restriction of $f$ to each simplex $\sigma$ has a Lipschitz constant $\frac{1}{\delta}$.
  Then $f$ is a Lipschitz map by Lemma \ref{lemma_locallyLipschitz}.
\end{proof}

\begin{theorem}[Uniform simplicial approximation]
  \label{theorem_uniformsimplicialapproximation}
  Let $f\colon K\to K'$ be a uniformly continuous map between uniform simplicial complexes.
  Then for sufficiently large $m$, the composition
  \[
    \operatorname{sd}^mK\cong K\xrightarrow{f}K'
  \]
  with the canonical homeomorphism between $K$ and $\operatorname{sd}^mK$ is uniformly homotopic to a simplicial map.
  Moreover, if $f$ is Lipschitz on a subcomplex $J\subset K$, then we can take a homotopy from $f$ to a simplicial map to be Lipschitz on $J$.
\end{theorem}
\begin{proof}
  Consider the family $\mathcal{U}=\{\operatorname{st}(v)\}_{v\in L^0}$ of open stars around the vertices in $L$, which is an open covering with positive Lebesgue number.
  As $f$ is uniformly continuous, the pullback $f^{-1}\mathcal{U}$ of $\mathcal{U}$ is an open covering of $K$ with positive Lebesgue number by Lemma \ref{lemma_pullbackLebesguenumber}.
  For a sufficiently large $m$, the closed star of each vertex in $K$ is contained in $f^{-1}(\operatorname{st}(w_v))$ for some vertex $w_v\in L$.
  Then $g$ can be defined by setting $g(v)=w_v$ and extending linearly.
  Taking an affine embedding $\phi_\tau$ of each star $\operatorname{st}\tau$ into an Euclidean space, there exists a constant $A>0$ such that the inequalities
  \[
    A^{-1}|\phi_\tau(x)-\phi_\tau(y)|
    \le
    d_{K'}(x,y)
    \le
    A|\phi_\tau(x)-\phi_\tau(y)|
    \quad
    \text{and}
    \quad
    |\phi_\tau(x)|\le A
  \]
  hold for any simplex $\tau$ in $K'$ and $x,y\in\operatorname{st}(\tau)$.
  We can construct a homotopy $H$ from $f$ to $g$ by defining
  \[
    H(t,x)=\phi_{g(\sigma)}^{-1}((1-t)\phi_{g(\sigma)}(f(x))+t\phi_{g(\sigma)}(g(x)))
  \]
  if $x$ lies in the interior of a simplex $\sigma$, where $g(x)$ lies in the interior of $g(\sigma)$ and $f(x)$ in the open star of $g(\sigma)$.
  For more details about these facts, see \cite[Proof of Theorem 2C.1]{H}.
  By Lemmas \ref{lemma_subdivisionLipschitz} and \ref{lemma_simplicialmapLipschitz}, the map $g\colon K\to K'$ has a Lipschitz constant, say $L$.
  Since $K'$ is a uniform simplicial complex, there are only finitely many possible isomorphism classes of stars of simplices in $K'$.
  Suppose $d_{K'}(f(x),f(x'))<\epsilon$ for any $x,x'\in K$ with $d_K(x,x')<\delta$.
  When $x$ and $x'$ lie in the same simplex $\sigma$ and $d_K(x,x')<\delta$, since $f(x)$ and $f(x')$ lie in the open star of$g(\sigma)$, we have
  \begin{align*}
    &d_{K'}(H(t,x),H(t',x'))\\
    &
    \le
    A|(1-t)\phi_{g(\sigma)}(f(x))+t\phi_{g(\sigma)}(g(x))-(1-t')\phi_{g(\sigma)}(f(x'))-t'\phi_{g(\sigma)}(g(x'))|\\
    &
    \le
    A((1-t)|\phi_{g(\sigma)}(f(x))-\phi_{g(\sigma)}(f(x'))|
    +|t-t'||\phi_{g(\sigma)}(f(x'))|\\
    &\quad
    +t|\phi_{g(\sigma)}(g(x))-\phi_{g(\sigma)}(g(x'))|
    +|t-t'||\phi_{g(\sigma)}(g(x'))|)\\
    &
    \le
    A^2((1-t)d_{K'}(f(x),f(x'))
    +A|t-t'|
    +d_{K'}(g(x),g(x'))
    +A|t-t'|)\\
    &
    \le
    A^2\epsilon+A^2Ld_K(x,x')+A^3|t-t'|.
  \end{align*}
  Note that we can apply Lemma \ref{continuity} to the map $H$ from a uniform CW complex $[0,1]\times K$.
  Then, $H$ is a uniformly continuous homotopy from $f$ to $g$.
  When $f$ is Lipschitz on $J$, we can verify $H$ is Lipschitz on $J$ by a similar computation on simplices in $J$.
\end{proof}

As a corollary of this theorem, we obtain the equivalence of uniform homotopy theory and Lipschitz homotopy theory on uniform simplicial complexes as follows.

\begin{corollary}
  \label{corollary equivalence uniform Lipschitz}
  Let $K$ and $K'$ be uniform simplicial complexes.
  Then the following assertions hold.
  \begin{enumerate}
    \item Any uniformly continuous map between $K$ and $K'$ is uniformly homotopic to a Lipschitz map.
    \item If two Lipschitz maps between $K$ and $K'$ are uniformly homotopic, then these maps are homotopic through a Lipschitz homotopy.
  \end{enumerate}
\end{corollary}
\begin{proof}
  The assertion (1) immediately follows from Theorem \ref{theorem_uniformsimplicialapproximation}.
  For the assertion (2), suppose that $f,g\colon K\to K'$ are uniformly homotopic Lipschitz maps.
  Let $H\colon[0,1]\times K\to K'$ be a uniform homotopy from $f$ to $g$.
  Applying Theorem \ref{theorem_uniformsimplicialapproximation} to $H$ and the subcomplex $\{0,1\}\times K$, there exists a uniformly continuous homotopy $\mathcal{H}\colon[0,1]\times[0,1]\times K\to K'$ from $H$ to a Lipschitz homotopy $H'$ such that $\mathcal{H}$ is Lipschitz on $[0,1]\times\{0,1\}\times K$.
  A Lipschitz homotopy $H''\colon[0,1]\times K\to K'$ from $f$ to $g$ is obtained by
  \[
    H''(t,x)=
    \begin{cases}
      \mathcal{H}(3t,0,x) & (0\le t\le\frac{1}{3})\\
      \mathcal{H}(1,3t-1,x) & (\frac{1}{3}\le t\le\frac{2}{3})\\
      \mathcal{H}(3-3t,1,x) & (\frac{2}{3}\le t\le1).
    \end{cases}
  \]
\end{proof}

\end{document}